\documentclass[12pt]{amsart}

\usepackage{amssymb,amsbsy}
\usepackage{latexsym}
\usepackage{verbatim,color}
\usepackage{fullpage,euscript}
\usepackage[enableskew]{youngtab}

\usepackage{xspace}
\usepackage{xcolor}
\usepackage[colorlinks=true,linkcolor=blue,urlcolor=blue,citecolor=magenta]{hyperref}

\definecolor{my_color}{rgb}{0.5,0.4,0.1}
\definecolor{MIXT}{rgb}{0.7,0.1,0.2}

\numberwithin{equation}{section}
\input {cyracc.def}
 \font\tencyr=wncyr10 
\font\tencyi=wncyi10 
\font\tencysc=wncysc10 
\def\rus{\tencyr\cyracc}
\def\rusi{\tencyi\cyracc}
\def\rusc{\tencysc\cyracc}

\newtheorem{thm}{Theorem}[section]
\newtheorem{conj}[thm]{Conjecture}
\newtheorem{qtn}{Problem}
\newtheorem{lm}[thm]{Lemma}
\newtheorem{cl}[thm]{Corollary}
\newtheorem{prop}[thm]{Proposition}

\theoremstyle{remark}
\newtheorem{rmk}[thm]{Remark}

\theoremstyle{definition}
\newtheorem{ex}[thm]{Example}

\newtheorem{df}{Definition}
\newtheorem*{rema}{Remark}

\newcommand {\ah}{{\mathfrak a}}
\newcommand {\be}{{\mathfrak b}}
\newcommand {\ce}{{\mathfrak c}}

\newcommand {\fe}{{\mathfrak E}}

\newcommand {\g}{{\mathfrak g}}

\newcommand {\el}{{\mathfrak l}}

\newcommand {\p}{{\mathfrak p}}
\newcommand {\q}{{\mathfrak q}}

\newcommand {\te}{{\mathfrak t}}
\newcommand {\ut}{{\mathfrak u}}
\newcommand {\z}{{\mathfrak z}}

\newcommand {\gln}{{\mathfrak{gl}_n}}
\newcommand {\sln}{{\mathfrak{sl}_n}}
\newcommand {\slN}{{\mathfrak{sl}_N}}

\newcommand {\spn}{{\mathfrak {sp}}_{2n}}

\newcommand {\son}{{\mathfrak{so}_n}}
\newcommand {\sono}{{\mathfrak{so}_{2n+1}}}
\newcommand {\sone}{{\mathfrak{so}_{2n}}}

\newcommand {\co}{{\mathcal O}}

\newcommand {\BV}{{\mathbb V}}

\newcommand {\BW}{{\mathbb W}}

\newcommand {\md}{/\!\!/}
\newcommand {\isom}{\stackrel{\sim}{\longrightarrow}}

\newcommand {\eus}{\EuScript}

\newcommand {\esi}{\varepsilon}

\newcommand{\ap}{\alpha}

\newcommand{\curge}{\succcurlyeq}
\newcommand{\curle}{\preccurlyeq}
\renewcommand{\le}{\leqslant}
\renewcommand{\ge}{\geqslant}

\newcommand {\ind}{{\mathrm{ind\,}}}
\newcommand {\Lie}{{\mathrm{Lie\,}}}

\newcommand {\rk}{{\mathrm{rk\,}}}

\newcommand {\spe}{{\mathsf{Spec\,}}}

\newcommand {\trdeg}{{\mathrm{trdeg\,}}}

\newcommand {\beq}{\begin{equation}}
\newcommand {\eeq}{\end{equation}}

\newcommand {\GR}[2]{{\textrm{{\bf #1}}}_{#2}}

\newcommand {\ov}{\overline}
\newcommand {\un}{\underline}

\newcommand {\bbk}{\Bbbk}

\newcommand {\dah}{{\Delta_\ah}}
\newcommand {\dce}{{\Delta_\ce}}
\newcommand {\sah}{{\mathfrak S_\ah}}

\newcommand {\mybullet}{{\color{MIXT}\bullet}}

\begin{document}
\setlength{\parskip}{3pt plus 2pt minus 0pt}
\hfill {\scriptsize April 26, 2014}
\vskip1ex

\title[On the orbits of a Borel subgroup in abelian ideals]
{On the orbits of a Borel subgroup in abelian ideals}
\author[D.\,Panyushev]{Dmitri I.~Panyushev}
\address
{Institute for Information Transmission Problems of the R.A.S., 
\hfil\break\indent \quad Bol'shoi Karetnyi per. 19, Moscow 127994, Russia
}
\email{panyushev@iitp.ru}
\subjclass[2010]{14L30, 17B08, 20F55, 05E10}
\keywords{Root system, ad-nilpotent ideal, strongly orthogonal roots, the Pyasetskii duality}
\begin{abstract}
Let $B$ be a Borel subgroup of a semisimple algebraic group $G$, and let $\ah$ be an abelian ideal of 
$\be=\Lie(B)$. The ideal $\ah$ is  determined by certain subset $\dah$ of positive roots, and 
using $\dah$ we give an explicit classification of the $B$-orbits in $\ah$ and $\ah^*$. Our 
description visibly demonstrates that there are finitely many $B$-orbits in both cases.
Then we describe the Pyasetskii correspondence between the $B$-orbits in $\ah$ and $\ah^*$ and 
the invariant algebras $\bbk[\ah]^U$ and $\bbk[\ah^*]^U$, where $U={(B,B)}$. As an application, the 
number of $B$-orbits in the abelian nilradicals is computed. We also discuss related results of 
A.~Melnikov and others for classical groups and state a general conjecture on the closure and 
dimension of the $B$-orbits in the abelian nilradicals, which exploits a relationship between between 
$B$-orbits and  involutions in the Weyl group.
\end{abstract}
\maketitle  

\tableofcontents
\section*{Introduction}
\label{sect:intro}
\noindent
Let $G$ be a connected semisimple algebraic group with Lie algebra $\g$. Fix a Borel subgroup $B$
and a maximal torus $T\subset B$. Let $\ah$ be an abelian ideal of $\be=\Lie(B)$, i.e., $\ah\subset\be$, 
$[\be,\ah]\subset\ah$, and $[\ah,\ah]=0$. It is easily seen that $\ah\subset [\be,\be]=:\ut$ and hence 
$\ah$ is a sum of certain root spaces. Therefore $G{\cdot}\ah$ is the closure of a nilpotent $G$-orbit in 
$\g$. By \cite[Theorem\,2.3]{pr01}, $G{\cdot}\ah$ is the closure of a {\sl spherical\/} $G$-orbit. That result
is based on the characterisation of the spherical nilpotent $G$-orbits obtained in~\cite[(3.1)]{sph-Man}.
Consequently, the $B$-module $\ah$ has finitely many orbits (that is, the set $\ah/B$ is finite).
By a general result of Pyasetskii~\cite{pi75}, this is equivalent to that, for the dual $B$-module $\ah^*$, 
the set $\ah^*/B$ is finite.

In this article, a direct approach to the 
study of $B$-orbits in $\ah$ is provided. We prove that $\ah/B$ is finite, without using the sphericity of
$G{\cdot}\ah$, and point out a representative for each $B$-orbit in $\ah$. Describing $B$-orbits in abelian ideals immediately reduces to {\bf simple} Lie algebras 
and, from now on, we assume that $\g$ is simple. 
Let $\dah$ be the subset of positive roots corresponding to $\ah$. We say that $\eus S\subset\dah$ 
is {\it strongly orthogonal}, if each pair of roots in $\eus S$ is strongly orthogonal in the usual sense, cf.
Definition~\ref{def1} below. We establish a natural bijection between $\ah/B$ and the set, $\sah$, of all 
strongly orthogonal subsets of $\dah$. 
Namely, let us fix nonzero root vectors $\{e_\gamma\}_{\gamma\in\Delta^+}$ and, for any
$\eus S\in\sah$, set $e_\eus S=\sum_{\gamma\in\eus S} e_\gamma\in \ah$. Then 
$\{e_\eus S\}_{\eus S\in\sah}$ is a complete set of representatives of $B$-orbits in $\ah$
(Theorem~\ref{thm:B-orbits-ah}).
Quite independently, without using Pyasetskii's result~\cite{pi75}, we obtain a similar set of representatives for the $B$-orbits in $\ah^*$, also parameterised by $\sah$ (Theorem~\ref{thm:b-orb-dual}).
Both classifications rely on the following simple
observation. Let $\gamma_1,\gamma_2$ be strongly orthogonal roots in $\dah$. 
Set $\Delta^{(+)}_{\gamma_i}=\{\delta\in\Delta^+\mid \gamma_i+\delta\in\dah\}$ and
$\Delta^{(-)}_{\gamma_i}=\{\delta\in\Delta^+\mid \gamma_i-\delta\in\dah\}$. Then
$\Delta^{(+)}_{\gamma_1}\cap \Delta^{(+)}_{\gamma_2}=\varnothing$ and
$\Delta^{(-)}_{\gamma_1}\cap \Delta^{(-)}_{\gamma_2}=\varnothing$ (Lemma~\ref{lm:simple-main}).

For $\eus S\in\sah$, let $\co_\eus S$ (resp. $\co_\eus S^*$) denote the corresponding $B$-orbit in
$\ah$ (resp. $\ah^*$). We point out two sets $\eus C^{l},\eus C^u\in\sah$ that give rise to the dense $B$-orbits in $\ah$ and $\ah^*$, respectively. Furthermore, Pyasetskii's theory yields a natural one-to-one correspondence (duality) between $\ah/B$ and $\ah^*/B$ (see~\ref{subs:FO}), and we explicitly describe it. More precisely, given $\co_\eus S\in\ah/B$, let $(\co_\eus S)^\vee$ denote the
Pyasetskii dual orbit in $\ah^*$. Then $(\co_\eus S)^\vee=\co^*_{\eus S^\vee}$ for some $\eus S^\vee\in\sah$, 
and we determine $\eus S^\vee$ via $\eus S$, see Theorem~\ref{thm:Pyasetskii-general}.

Set $U=(B,B)$. Since both $\ah$ and $\ah^*$ contain dense $B$-orbits, it follows from \cite[\S\,4, Prop.\,5]{sk77} that
the algebras of $U$-invariants $\bbk[\ah]^U$ and $\bbk[\ah^*]^U$ are polynomial. We show that their
Krull dimensions equal $\#\eus C^{l}$ and $\#\eus C^u$,  respectively. This also implies that the number
of $B$-orbits of codimension $1$ in $\ah$ (resp. $\ah^*$) equals $\#\eus C^{l}$ (resp. $\#\eus C^u$). 
Moreover, the description of  $\bbk[\ah^*]^U$ holds true upon replacing $\ah$ with an arbitrary 
$\be$-ideal $\ce\subset\ut$, see Section~\ref{sect:U-inv}.

Let $P$ a standard parabolic subgroup of $G$ with $\Lie(P)=\p$ and $\p^u$ the nilradical of  $\p$. 
The {\it abelian nilradicals}\/ (=\textsf{ANR}) $\p^u$ yield the most interesting examples of abelian ideals 
of $\be$, and for any such $\p^u$ we compute the total number of $B$-orbits and also the number of
orbits $\co_\eus S$ such that $\#\eus S$ is a prescribed integer,
see Section~\ref{sect:ab-nilrad}. 

It is a fundamental problem to describe the closures of $B$-orbits in $\ah$ and $\ah^*$, i.e., the natural 
poset structure on $\ah/B$ and $\ah^*/B$. In general, these two posets are rather unrelated and one 
has two different problems. (The Pyasetskii duality tends to behave as a poset anti-isomorphism, 
but only to some extent!)  Although no general solution to either of the problems is known,  
we have a conjecture on the case in which $\ah$ is an \textsf{ANR} $\p^u$. Let $L$ denote the
standard Levi subgroup of $P$. Then $\p^u$ and $(\p^u)^*$ are dual $L$-modules and the $B$-orbits in
$\p^u$ coincide with the $B\cap L$-orbits. This implies that the posets $\p^u/B$ and $(\p^u)^*/B$ 
are naturally isomorphic, and it is more convenient to state our conjecture for $B$-orbits in $(\p^u)^*$.
To any $\eus S\in\sah$ one associates the involution $\sigma_\eus S\in W$ that is the product 
of reflections corresponding to all roots in $\eus S$. Let $\ell$ be the length function on $W$. For any
$w\in W$, we regard $1-w$ as an endomorphism of $\te$.
It is well-known that $\rk(1-w)$ is the minimal length for presentations of $w$ as a product of 
{\bf arbitrary} reflections in $W$, which is also called the {\it absolute length} of $w$.
For $\ah=\p^u$, we conjecture that 
(i) $\co^*_{\eus S'}\subset
\ov{\co^*_{\eus S}}$ if and only if $\sigma_{\eus S'}\le \sigma_\eus S$ w.r.t the Bruhat order
and (ii)  $\dim \co^*_\eus S=\frac{\ell(\sigma_\eus S)+\rk(1-\sigma_\eus S)}{2}$, see 
Conjecture~\ref{conj:bruhat-closure-ahs}.  But both assertions are false 
for arbitrary maximal abelian ideals, see Example~\ref{ex:D4}.

We also give in Section~\ref{sect:closure} an account on related results for classical algebras $\g$ that 
are due to Melnikov and others~\cite{bagno-chern, chern11,ign,Anna00,Anna06}. In fact, our approach 
provides a unified treatment for problems studied independently for different series of simple Lie 
algebras.

\noindent
{\sl \un{Main notation}.} 
The ground field $\bbk$ is algebraically closed and of characteristic zero.\\
$\Delta$ is the set of roots of $(G,T)$, $\Delta^+$ is the set of positive roots corresponding to
$U$, and $\Pi$ is the set of simple roots in $\Delta^+$;  $W$ is the Weyl group of $(G,T)$ and $\theta$ is the {\it highest root\/} in $\Delta^+$.
For $\gamma\in\Delta^+$, $U_\gamma$ is the root subgroup of $U$ and 
$\ut_\gamma=\Lie(U_\gamma)$. Then $\ut=\bigoplus_{\gamma\in\Delta^+}\ut_\gamma$.

If an algebraic group $Q$ acts on an irreducible affine variety $X$, then $\bbk[X]^Q$ 
is the algebra of $Q$-invariant regular functions on $X$ and $\bbk(X)^Q$
is the field of $Q$-invariant rational functions. If $\bbk[X]^Q$
is finitely generated, then $X\md Q:=\spe \bbk[X]^Q$.

If $x\in X$, then $Q^x$ is the stabiliser of $x$ in $Q$ and $\q^x=\Lie(Q^x)$.

\vskip1ex
\noindent
{\small  {\bf Acknowledgements.} Part of this work was done while I was able to use
rich facilities of the Max-Planck Institut f\"ur Mathematik (Bonn).}

\section{Preliminaries on linear actions with finitely many orbits and abelian ideals} 
\label{sect:prelim}

\noindent
\subsection{Representations with finitely many orbits}    \label{subs:FO}
Let $\nu: Q\to GL(\BV)$ be a representation of a connected algebraic group such that 
$\#(\BV/Q) < \infty$. By \cite{pi75}, the dual representation also enjoys this property. 
More precisely, Pyasetskii provides a natural bijection between two sets 
of $Q$-orbits and thereby obtains that $\#(\BV/Q)=\#(\BV^*/Q)$. It works as follows. Consider the moment map $\mu: \BV\times \BV^* \to \q^*$ and its reduced 
zero-fibre $\mu^{-1}(0)_{red}=:\fe$.
Under the assumption that $\#(\BV/Q)< \infty$, $\fe$ is a ($Q$-stable) variety of pure dimension $\dim \BV$,
and the set of irreducible components of $\fe$, $\text{Irr}(\fe)$, is in a one-to-one correspondence with 
the set of $Q$-orbits in $\BV$, or with the set of $Q$-orbits in $\BV^*$. 
Namely, let $p_1:\BV\times \BV^*\to \BV$ and $p_2:\BV\times \BV^*\to \BV^*$ be two projections.
If $\fe_i\in \text{Irr}(\fe)$, then $p_i(\fe)$, $i=1,2$, contains a dense $Q$-orbit and 
this yields the bijection between $\BV/Q$ and $\BV^*/Q$, which is called the {\it Pyasetskii duality (correspondence)}. 
It is obtained as  the composition of natural bijections:
\[
     \BV/Q  \stackrel{1:1}{\longleftrightarrow}\text{Irr}(\fe) \stackrel{1:1}{\longleftrightarrow} \BV^*/Q .
\]
For $\co\in \BV/Q$,  the Pyasetskii dual $Q$-orbit in $\BV^*$ is denoted by $\co^\vee$. 
The passage $\co\mapsto \co^\vee$ is described directly as follows. For $v\in\co$, let
$(\q{\cdot}v)^\perp$ denote the annihilator of $\q{\cdot}v$ in $\BV^*$. 
Then $Q{\cdot}(\q{\cdot}v)^\perp$ is irreducible and contains the dense $Q$-orbit, which is $\co^\vee$.
This also shows that the  component $\fe_i$ corresponding to 
$\co_i\in \BV/Q$ (or $\co_i^\vee\in \BV^*/Q$) is the closure of the conormal bundle of $\co_i$ 
(or $\co_i^\vee$). Below is a slight extension of the Pyasetskii result.

\begin{lm}   \label{lem:warming-up}
If\/ $\BV$ and $\BW$ are $Q$-modules and  $(\BV\oplus\BW)/Q$ is finite, then there is a natural one-to-one correspondence between $(\BV\oplus\BW)/Q$ and 
$(\BV\oplus\BW^*)/Q$. In particular, $\#(\BV\oplus\BW^*)/Q < \infty $.
\end{lm}
\begin{proof} The moment map $(\BV\oplus\BW)\oplus (\BV^*\oplus\BW^*)\to \q^*$ associated with the
$Q$-module $\BV\oplus\BW$ can also be regarded as the moment map for  the
$Q$-module $\BV\oplus\BW^*$.
\end{proof}

\subsection{Ad-nilpotent and abelian ideals of $\be$}
Let $\ce$ be a $B$-stable subspace of $\ut$. Then $\ce$ is an ideal of $\be$ that consists of 
{\sl ad}-nilpotent elements, and we  say that $\ce$ is an {\sl ad}-nilpotent
ideal of $\be$. Every {\sl ad}-nilpotent ideal is a sum of root spaces, i.e., 
$\ce=\bigoplus_{\gamma\in\dce}\g_\gamma$, where $\dce \subset\Delta^+$, and
$\dce$ is called a {\it combinatorial ideal} in $\Delta^+$. Abusing the language, we will often omit the 
word  `combinatorial' and refer to $\dce$ as an ideal, too.
If $[\ce,\ce]=0$, then $\ce$ is an {\it abelian ideal\/} of $\be$ (and $\dce$ is a {\it combinatorial abelian 
ideal}), and we use the letter `$\ah$' for such ideals.  That is, $\ah$ is {\bf always} an abelian ideal of 
$\be$. Although we are primarily interested in $B$-orbits related to abelian ideals and their duals, we 
also obtain some results that hold for arbitrary {\sl ad}-nilpotent ideals. The combinatorial 
ideals $\dce$ has the following characteristic property:
\begin{itemize}
\item if $\gamma\in\dce$, $\mu\in\Delta^+$, and $\gamma+\mu\in\Delta^+$, then
$\gamma+\mu\in\dce$, 
\end{itemize}
and the abelian ideals $\dah$  have the additional characteristic property
\begin{itemize}
\item if $\gamma_1,\gamma_2\in\dah$, then $\gamma_1+\gamma_2\not\in\Delta$ .
\end{itemize}
We equip  $\Delta$  with the usual partial ordering `$\curle$'.
This means that $\mu\curle\nu$ if $\nu-\mu$ is a non-negative integral linear combination
of simple roots. 
For any $M\subset\Delta^+$, let $\min(M)$ (resp. $\max(M)$) denote the set of its minimal (resp. 
maximal) elements with respect to `$\curle$'. A combinatorial ideal $\dce$ is fully determined by 
$\min(\dce)$. If $M\subset \Delta^+$, then $[M]$ is the combinatorial ideal generated by $M$, i.e.,
$[M]:=\{\gamma\in \Delta^+\mid \gamma\curge \nu\text{ for some } \nu\in M\}$. Then 
$\min([M])=\min(M)\subset M$ and $\bigoplus_{\gamma\in [M]}\ut_\gamma$ is the minimal $B$-stable 
subspace containing all $\ut_\gamma$, $\gamma\in M$.
If $M\subset \dah$ for some abelian ideal $\ah$, then $[M]$ is also abelian.

\begin{df}  \label{def1}
Two different roots $\gamma_1,\gamma_2$ are said to be {\it strongly orthogonal}, if neither
of $\gamma_1\pm\gamma_2$ 
is a root. In this case, one has $(\gamma_1,\gamma_2)=0$, where $(\ ,\ )$ is a $W$-invariant scalar product in $\te^*$.
A subset $S\subset \Delta$ is {\it strongly orthogonal}, if each pair of roots in $S$ is strongly orthogonal.
\end{df}
\begin{rema}
If $\Delta$ is simply-laced, then `strongly orthogonal' is the same as `orthogonal'. Therefore, we omit
the word `strongly' in our further examples related to the {\bf ADE}-cases.
\end{rema}
Our results on $B$-orbits in $\ah$ and $\ah^*$  rely on the following simple observation.

\begin{lm}   \label{lm:simple-main}
Suppose that $\gamma_1,\gamma_2\in\dah$ are strongly orthogonal. 
\begin{itemize}
\item[\sf (i)] \ If\/ $\gamma_1+\delta\in\Delta^+$ \ for some $\delta\in\Delta^+$, then $\gamma_2+\delta\not\in\Delta^+$;
\item[\sf (ii)] \ If\/ $\gamma_1-\delta\in\dah$ \ for some $\delta\in \Delta^+$, then $\gamma_2-\delta\not\in\dah$.
\end{itemize}
\end{lm}
\begin{proof}
(i) Assume that $\gamma_2+\delta\in\Delta^+$ as well. Excluding the case of $\GR{G}{2}$, which is easy 
to handle directly (see below), we then have $(\gamma_1,\delta)\le 0$ and $(\gamma_2,\delta)\le 0$.
\\ \indent
1$^o$. \ Suppose that one of these scalar products is negative, say $(\gamma_2,\delta)< 0$.
Then $(\gamma_1+\delta,\gamma_2)<0$ and hence $\gamma_1+\delta+\gamma_2\in \dah$, 
which contradicts the fact that $\ah$ is abelian.
\\ \indent
2$^o$. \ If $(\gamma_1,\delta)=(\gamma_2,\delta)=0$, then $(\delta+\gamma_1,\delta+\gamma_2)=
(\delta,\delta)>0$. Then $\gamma_1-\gamma_2\in\Delta$, which contradicts the strong orthogonality.

(ii) If both $\gamma_1-\delta$ and $\gamma_2-\delta$ belong to $\dah$, then these two roots are 
strongly orthogonal. Then applying part (i) to them yields a contradiction.
\end{proof}

\begin{ex}   \label{ex:G2}
For $\g$ of type $\GR{G}{2}$, the unique maximal abelian ideal is $3$-dimensional. If $\Pi=\{\ap,\beta\}$,
where $\ap$ is short, then $\dah=\{2\ap+\beta,\,3\ap+\beta,\,3\ap+2\beta\}$. Here $\dah$ contains no pairs of orthogonal roots!
\end{ex}

\section{Classification of $B$-orbits in $\ah$}
\label{sect:B-ah}

\noindent
For every $\gamma\in\Delta^+$, we fix a nonzero root vector $e_\gamma\in \ut_\gamma$. 
Let $\ah$ be an abelian ideal of $\be$ and $\dah$ the corresponding set of positive roots.
For a nonempty $M\subset \dah$, we set
$e_M:=\sum_{\gamma\in M}e_\gamma \in \ah$. If $M=\varnothing$, then $e_\varnothing=0$.

\begin{lm}   \label{lm:linear-span}
The linear span of $B{\cdot}e_M$,  $\langle B{\cdot}e_M\rangle$, equals 
$\bigoplus_{\gamma\in [M]}\ut_\gamma \subset\ah$.
\end{lm}
\begin{proof}
It follows from the construction that $\bigoplus_{\gamma\in [M]}\ut_\gamma$ is the smallest 
$B$-stable subspace containing $e_M$.
\end{proof}

For the future use, we record the following obvious fact:
\\ \textbullet \quad
{\it If $M\subset \dah$, then either of\/ $\min(M)$ and $\max(M)$ is a strongly orthogonal set}.
\par
Let $\sah$ denote the set of all strongly orthogonal subsets of $\dah$.

\begin{thm}   \label{thm:B-orbits-ah}
There is a natural one-to-one correspondence
\[
    \ah/B  \stackrel{1:1}{\longleftrightarrow} \{\eus S\subset\dah \mid \eus S\ \text{ is strongly orthogonal}\}
    =\sah.
\] 
This correspondence takes $\eus S$ to the orbit $\co_{\eus S}:=B{\cdot}e_{\eus S}\subset \ah$.
\end{thm}
\begin{proof}
Our proof consists of two parts (assertions):

(a) \ For any $v\in \ah$, the orbit $B{\cdot}v$ contains an element of the form 
$e_\eus S$ for some $\eus S\in \sah$.

(b) \ If $B{\cdot}e_{\eus S}=B{\cdot}e_{\eus S'}$, then $\eus S=\eus S'$.

{\sl Part (a)}. For $v=\sum_{\gamma\in\dah} a_\gamma e_\gamma\in \ah$, we set 
$\text{supp}(v):=\{\gamma\in\dah\mid a_\gamma\ne 0\}$. We describe below a reduction procedure 
that gradually transforms $v$  into $\hat v \in U{\cdot}v$ such that $\text{supp}(\hat v)$ is strongly 
orthogonal. Consider the strongly orthogonal set  
$\Gamma=\min(\text{supp}(v))=:\{\gamma_1,\dots,\gamma_k\}$ and
\[
     \Delta^{(+)}_{\gamma_i}=\{\delta\in\Delta^+\mid \gamma_i+\delta\in \Delta^+\} .
\]
By Lemma~\ref{lm:simple-main}(i), we have 
$\Delta^{(+)}_{\gamma_i}\cap\Delta^{(+)}_{\gamma_j}=\varnothing$ if $i\ne j$. (Note, however, that the 
sets $\gamma_i+\Delta^{(+)}_{\gamma_i}$, $i=1,\dots,k$, are not necessarily disjoint.) 
This implies that using root subgroups $U_\delta\subset U$ with 
$\delta\in \bigsqcup_{i=1}^k \Delta^{(+)}_{\gamma_i}$, one can consecutively get rid of all root 
summands of $v$ whose  roots belong to  
\[
  \eus M_\Gamma:=\bigcup_{i=1}^k (\gamma_i+\Delta^{(+)}_{\gamma_i})=
  \bigl(\bigcup_{i=1}^k (\gamma_i+\Delta^+)\bigr)\cap\Delta^+=:(\Gamma+\Delta^+)\cap\Delta^+
  =(\Gamma+\Delta^+)\cap\dah .
\]
More precisely, one can write
\[
   v=\sum_{\gamma\in\Gamma}a_\gamma e_\gamma
   + \sum_{\nu\in \eus M_{\Gamma}} a_\nu e_\nu + \tilde v  
   \quad (a_{\gamma}{\ne\,}0 \text{ for } \gamma\in\Gamma) ,
\]
where $\tilde v\in\ah$ represents the sum related to the roots outside 
$\Gamma\sqcup \eus M_\Gamma$.\\
Given $\nu\in \min(\eus M_\Gamma\cap \text{supp}(v))$, there are $\gamma\in\Gamma$ and
$\delta\in\Delta^+$ such that $\nu=\gamma+\delta$. Then there exists a unique
$\tilde u\in U_{\delta}$ such that $\nu\not\in\text{supp}(\tilde u{\cdot}v)$
(i.e., we kill the summand with $e_{\nu}$). By Lemma~\ref{lm:simple-main}(i), this transformation does 
not affect the $\Gamma$-group of summands. It may change other summands in the $\eus M_\Gamma$-group and also change $\tilde v$, but the important thing is that $\eus M_\Gamma\cap \text{supp}(\tilde u{\cdot}v)$ 
generates a smaller ideal than $\eus M_\Gamma\cap \text{supp}(v)$ does. Continuing this way, we 
eventually kill all summands in the  $\eus M_\Gamma$-group.
In other words, there is $u\in U$ such that 
\[
   u{\cdot}v=\sum_{\gamma\in\Gamma}a_\gamma e_\gamma+v' ,
\]
where 
$\text{supp}(v')$ is strongly orthogonal to $\Gamma$. Then we set 
$\Gamma'=\Gamma\cup\min(\text{supp}(v'))$ and play the same 
game with $v'$ and the strongly orthogonal set $\Gamma'$. Again, in view of Lemma~\ref{lm:simple-main}(i), making further reductions with $v'$, 
does not change the sum $\sum_{\gamma\in\Gamma'}a_\gamma e_\gamma$, and we can kill all the summands with weights in $\eus M_{\Gamma'}\supset \eus M_\Gamma$.
Eventually, we obtain a vector 
$\displaystyle\sum_{\gamma\in\eus S}a_{\gamma} e_{\gamma}\in U{\cdot}v$, 
where the set $\eus S$ is strongly orthogonal, $\eus S\supset\Gamma'\supset\Gamma$, and all 
coefficients $\{a_{\gamma}\}$  are nonzero.
Finally, since 
the roots in $\eus S$ are linearly independent,  we can make all 
$a_{\gamma}=1$ using a suitable element of $T$. 

{\sl Part (b)}.  Assume that $\eus S,\eus S'\in\sah$ and $e_{\eus S} \underset{B}{\sim} e_{\eus S'}$. 
\\
Clearly, $\langle B{\cdot}e_{\eus S}\rangle=\langle B{\cdot}e_{\eus S'} \rangle=:\hat\ah$, and this is an 
abelian ideal inside $\ah$. If $\Gamma=\min(\Delta_{\hat\ah})$, then $\Gamma \subset \eus S\cap \eus S'$ 
in view of Lemma~\ref{lm:linear-span}. Set $\tilde{\eus S}=\eus S\setminus \Gamma$,
$\tilde{\eus S}'=\eus S'\setminus \Gamma$ and consider the corresponding decompositions
\[
   e_{\eus S}=e_\Gamma+\tilde e, \quad e_{\eus S'}=e_\Gamma+\tilde e' \quad (\tilde e=e_{\tilde{\eus S}},
   \ \tilde e'=e_{\tilde{\eus S}'}) .
\]
Suppose that $b{\cdot}e_{\eus S}=e_{\eus S'}$ and $b=t^{-1}u$ with $t\in T,\ u\in U$. Then
\[
    u{\cdot}e_\Gamma+u{\cdot}\tilde e=t{\cdot}e_\Gamma+t{\cdot}\tilde e' .
\]
If $u{\cdot}e_\Gamma\ne e_\Gamma$, then $u{\cdot}e_\Gamma$ has nonzero summands corresponding to some roots in $\eus M_\Gamma$, which cannot occur in the right-hand side.
(For, if $u=\exp(n)$, $n\in\ut$, then $u{\cdot}e_\Gamma=e_\Gamma+[n,e_\Gamma]+\frac{1}{2}[n,[n,e_\Gamma]]+\dots$.) Hence $u{\cdot}e_\Gamma=e_\Gamma$ and $t{\cdot}e_\Gamma=e_\Gamma$. Therefore,
$\tilde e$ and $\tilde e'$ are $B^{e_\Gamma}$-conjugate. Since $\tilde{\eus S}$ and $\tilde{\eus S}'$ are
strongly orthogonal subsets in a smaller combinatorial ideal, 
arguing by downward induction on $\dim\hat\ah$, we conclude that $\tilde{\eus S}=\tilde{\eus S}'$.
Thus, $\eus S=\eus S'$.
\end{proof}

Because the set $\sah$ is clearly finite, we obtain
\begin{cl}    \label{cor:finite}
 The set of $B$-orbits in $\ah$, $\ah/B$,  is finite.
\end{cl}

Along with the bijection $\ah/B\longleftrightarrow \sah$, we produced a representative in every 
$B$-orbit. We say that $e_\eus S$ is the {\it canonical} representative in $\co_\eus S$ (it depends only 
on the normalisation of root vectors $e_\gamma$, $\gamma\in\dah$).  As a by-product of 
Lemma~\ref{lm:simple-main} and our proof of Theorem~\ref{thm:B-orbits-ah}, one obtains the following 
description of the tangent space of $\co_\eus S$ at $e_\eus S$.

\begin{prop}  \label{prop:tangent-sp}
For $\eus S\in\sah$, the tangent space $[\be,e_\eus S]$ is $T$-stable and the corresponding set of roots is
$\eus S\cup \eus M_{\eus S}$, where $\eus M_{\eus S}=(\eus S+\Delta^+)\cap \Delta^+$. More precisely,
$\eus S$ is the set of roots of\/ $[\te,e_\eus S]$ and $\eus M_{\eus S}$ is the set of roots of\/ $[\ut,e_\eus S]$.
In particular, $\dim\co_\eus S= \#(\eus S)+ \#(\eus M_{\eus S})$.
\end{prop}

\noindent
Our next goal is to describe the strongly orthogonal set in $\dah$ corresponding to the dense $B$-orbit 
in $\ah$. We define the {\it lower-canonical set\/} $\eus C^{l}\subset\dah$ inductively, as follows.
We begin with $\Gamma_{1}=\min(\dah)$ and put $\eus M_1=(\Gamma_1+\Delta^+)\cap \Delta^+$.
If $\Gamma_i$ and $\eus M_i$ are already constructed, then we set
$\Gamma_{i+1}=\min\bigl(\dah\setminus(\bigcup_{j=1}^i(\Gamma_j\cup\eus M_j)\bigr)$ and
$\eus M_{i+1}=(\Gamma_{i+1}+\Delta^+)\cap \Delta^+$. Eventually, we get $\Gamma_m=\varnothing$ 
and define $\eus C^{l}=\Gamma_{1}\cup \Gamma_{2}\cup\ldots\cup \Gamma_{m-1}$. By the
construction, the difference of two roots in $\eus C^{l}$ is not a root; and since we are inside an abelian 
ideal, the sum of two roots is never a root. Thus, $\eus C^{l}$ is strongly orthogonal. Whenever we
wish  to stress that $\eus C^{l}$ is determined by $\ah$, we write $\eus C^{l}_\ah$ for it.

\begin{lm}    \label{lm:dense-ah}
The lower-canonical set $\eus C^{l}\in\sah$ gives rise to the dense $B$-orbit in $\ah$.
\end{lm}
\begin{proof}
The above construction of $\eus C^{l}$ as a union 
$\Gamma_{1}\cup \Gamma_{2}\cup\ldots\cup \Gamma_{m-1}$ shows that any 
$\mu\in \dah\setminus \eus C^{l}$ belongs to a unique $\eus M_{i}$. Therefore, there exists 
$\gamma_\mu\in \Gamma_{i}\subset\eus C^{l}$ such that 
$\mu-\gamma_\mu\in \Delta^+$. By Lemma~\ref{lm:simple-main}(i), all the roots
$\mu-\gamma_\mu$ are different since ${\eus C^{l}}$ is strongly orthogonal. This implies that
$[\be,e_{\eus C^{l}}]=\ah$. Hence $B{\cdot}e_{\eus C^{l}}$ is dense in $\ah$.
\end{proof}

\begin{rema}
It is \un{not} true that for each $\mu$ there exists a {\sl unique\/} $\gamma_\mu$. We just pick one $\gamma_\mu$ with the required property.
\end{rema}

\begin{ex}   \label{ex:sln-ah}
For $\g=\sln$, we take $\be=\be(\sln)$ to be the algebra of traceless upper-triangular matrices. 
We stick to the usual matrix interpretation, hence $\ut$ is represented by the right-justified Young 
diagram $(n-1,\dots,2,1)$. See below the diagram for $n=5$:
\[
    \Yvcentermath1 
\ut \sim \tiny\young(\hfil\hfil\hfil\hfil,:\hfil\hfil\hfil,::\hfil\hfil,:::\hfil) \ .
\]
Each box of the diagram represents a positive root, with usual $\esi$-notation. For instance, the
north-east box is the highest root $\theta=\esi_1-\esi_n$. The
{\sl ad}-nilpotent ideals of $\be$ correspond to the right-justified Young diagrams that fit
inside the above diagram of $\ut$. Then the maximal abelian ideals of $\be$ are the nilradicals of 
maximal parabolic subalgebras, i.e., these are the rectangles
$(\underbrace{k,\dots,k}_{n-k})=:(k^{n-k})$, $k=1,\dots,n-1$. The maximal abelian ideals for
$n=5$ are depicted below: 
\[
  \Yvcentermath1
  \tiny\young(\hfil\hfil\hfil\hfil)\ , \quad \tiny\young(\hfil\hfil\hfil,\hfil\hfil\hfil)\ , 
  \quad \tiny\young(\hfil\hfil,\hfil\hfil,\hfil\hfil)\ , \quad \tiny\young(\hfil,\hfil,\hfil,\hfill) \ .
\]
(We do not draw the boxes outside the ideals!) An arbitrary abelian ideal $\ah$ 
corresponds to a diagram that fits inside one of such rectangles. Then $\min(\dah)$ is the set of
south-west corners of the diagram. Furthermore, for any $\gamma\in\Delta^+$, the
set $\{\gamma\}\cup \bigl( (\gamma+\Delta^+)\cap \Delta^+\bigr)$ is the hook with south-west corner
$\gamma$.
\\ \indent
For instance, consider the abelian ideal $\ah$ in $\be(\mathfrak{sl}_n)$, $n\ge 6$, with rows $(3,3,1)$. That is,
$\ah \sim \Yvcentermath1 
\tiny\young(\hfil\hfil\hfil,\hfil\hfil\hfil,::\hfil)$~. Here $\dah=\{\esi_1-\esi_{n-2},\esi_1-\esi_{n-1},\esi_1-\esi_{n},\esi_2-\esi_{n-2},\esi_2-\esi_{n-1},\esi_2-\esi_{n},\esi_3-\esi_{n}\}$ and $\Gamma_1=\{\esi_2-\esi_{n-2},
\esi_3-\esi_{n}\}$. The following diagram depicts $\Gamma_1\cup\eus M_1$, i.e.,
the union of hooks through $\Gamma_1$: \ 
$\Yvcentermath1 
\tiny\young(\ast\hfil\ast,\mybullet\ast\ast,::\mybullet)$ . (The roots in $\Gamma_1$ are denoted by bullets.)
The only remaining box is $\esi_1-\esi_{n-1}$, and this 
gives $\Gamma_2$. Thus, $\eus C^{l}$ is represented by the diagram: \ 
$\Yvcentermath1 
\tiny\young(\hfil\mybullet\hfil,\mybullet\hfil\hfil,::\mybullet)$ . It is not hard to compute that $\#(\ah/B)=20$ 
in this example.
\end{ex}

\section{Classification of $B$-orbits in $\ah^*$ and the Pyasetskii duality}
\label{sect:B-ah-s}

\noindent
For any {\sl ad}-nilpotent $\be$-ideal $\ce\subset\ut$, we think of the $B$-module $\ce^*$ as the quotient $\g/\ce^\perp$, where 
$\ce^\perp$ is the ortho\-com\-plement of $\ce$ in $\g$ with respect to the Killing form. The set of 
$T$-weights in $\ce^*$ is $-\dce$, and we fix a nonzero weight vector $\xi_{-\gamma}\in\ce^*$
for any $\gamma\in\dce$. (The $T$-weight of $\xi_{-\gamma}$ is $-\gamma$.)
It is convenient to choose roots vectors $e_{-\gamma}\in\g$, $\gamma\in\Delta^+$, and then define
$\xi_{-\gamma}$ as the image of $e_{-\gamma}$ in $\g/\ce^\perp$. This yields a choice of weight vectors in all $\ce^*$ that is compatible with the surjections $\ce^*_2\to \ce^*_1$ if $\ce_1\subset\ce_2$.

Although the set of weights of $\ah^*$ is $-\dah$, we prefer to think of it in terms of $\dah$. As this 
reverses the root order on the weights of $\ah^*$, we will have to consider the {\sl maximal\/} elements 
for subsets of $\dah$ in our constructions related to $\ah^*$.  Modulo such alterations, 
the classification of $B$-orbits in $\ah^*$ is being obtained in a fairly similar way.
For $M\subset\dah$, we set $\xi_{M}:=\sum_{\gamma\in M} \xi_{-\gamma}\in \ah^*$. (Again, if 
$M=\varnothing$, then $\xi_{M}=0$.) Let $I_M$ be the largest combinatorial ideal in $\dah$ such that 
$M\cap I_M=\varnothing$. Set 
\[
   \tilde M=\dah\setminus I_M \ \text{ and } \ \ah^*_{\tilde M}=\bigoplus_{\gamma\in \tilde M}\bbk \xi_{-\gamma}
   \subset \ah^* .
\]
Obviously, $M\subset \tilde M$ and $\xi_M\in \ah^*_{\tilde M}$.

\begin{lm}   \label{lm:linear-span-dual}
We have $\langle B{\cdot} \xi_{M}\rangle=\ah^*_{\tilde M}$.
\end{lm}
\begin{proof}
It is easily seen that $\ah^*_{\tilde M}$ is the smallest $B$-stable subspace of $\ah^*$ containing $\xi_{M}$. 
\end{proof}

Note that $\max(\dah)=\{\theta\}$, since $\g$ is assumed to be simple. Therefore, any non-empty
combinatorial ideal in $\dah$ contains $\theta$. This means that 
$\langle B{\cdot} \xi_{M}\rangle=\ah^*$ if and only if $I_M=\varnothing$ if and only if
$\theta\in M$. 

\begin{thm}    \label{thm:b-orb-dual}
There is a natural one-to-one correspondence
\[
   \ah^*/B  \stackrel{1:1}{\longleftrightarrow} \{\eus S\subset\dah \mid \eus S\ \text{ is strongly orthogonal}\}=\sah.
\] 
This correspondence takes $\eus S$ to the orbit $\co^*_{\eus S}:=B{\cdot}\xi_{\eus S}\subset \ah^*$.
\end{thm}
\begin{proof}
The argument is similar to that in Theorem~\ref{thm:B-orbits-ah}. One should use 
Lemma~\ref{lm:simple-main}(ii)  in place of Lemma~\ref{lm:simple-main}(i) and 
Lemma~\ref{lm:linear-span-dual} in place of Lemma~\ref{lm:linear-span}. For the reader convenience
and future reference, we outline the argument.

{\sl Part (a)}. 
For $\eta=\sum_{\gamma\in\dah}c_\gamma \xi_{-\gamma}\in\ah^*$, we consider 
$\text{supp}(\eta):=\{\gamma\in\dah\mid c_\gamma\ne 0\}$ and $\Gamma^*=\max(\text{supp}(\eta))$. 
Then we set
\[
  \eus M_{\Gamma^*}=\{\nu\in\dah \mid \nu =\gamma-\delta \ \text{ for some } \gamma\in\Gamma^*
  \ \& \ \delta\in\Delta^+\}=: (\Gamma^*-\Delta^+)\cap \dah .
\]
Accordingly, we write
$\eta=\sum_{\gamma\in\Gamma^*}c_\gamma\xi_{-\gamma}+
\sum_{\nu\in\eus M_{\Gamma^*}}c_\nu\xi_{-\nu}+ \tilde\eta$. For any $\gamma\in\Gamma^*$, consider
$\Delta^{(-)}_\gamma=\{\delta\in\Delta^+\mid \gamma-\delta\in\dah\}$.
By Lemma~\ref{lm:simple-main}(ii), the union $\bigcup_{\gamma\in\Gamma^*}\Delta^{(-)}_\gamma$ is
disjoint. Therefore,
using root subgroups $U_\delta$ with $\delta$ in this union, we may gradually kill the whole
$\eus M_{\Gamma^*}$-group of summands for $\eta$, without affecting the $\Gamma^*$-group of summands.
That is, there is $u\in U$ such that 
\[
  u{\cdot}\eta=\sum_{\gamma\in\Gamma^*}c_\gamma\xi_{-\gamma}+\eta' 
\]
and each root in $\text{supp}(\eta')$ is strongly orthogonal to $\Gamma^*$, and so on. Eventually we
obtain a representative in $U{\cdot}\eta$ whose support is strongly orthogonal.

{\sl Part (b)} is similar to the respective part in the proof of Theorem~\ref{thm:B-orbits-ah}.
\end{proof}

\begin{rmk}    \label{rem:zavisit-ahs}
If $\ah_1\subset\ah_2$ are two abelian ideals and $\eus S\in \mathfrak S_{\ah_1}\subset
\mathfrak S_{\ah_2}$, then the $B$-orbit $\co_\eus S\subset\ah_1$ is also a $B$-orbit in $\ah_2$. That 
is, the notation $\co_\eus S$ is unambiguous and can be used with any abelian ideal $\ah$ such that 
$\eus S\in\sah$. But this is not the case for the $B$-orbits in the dual spaces! The orbit $\co^*_\eus S$ depends on the ambient space $\ah^*$. If we write temporarily 
$\co^*_{\eus S,i}\subset \ah^*_i$, then the surjection $p:\ah^*_2\to \ah^*_1$ takes
$\co^*_{{\eus S},2}$ to $\co^*_{{\eus S},1}$, and the corresponding orbit dimensions are usually different.
\end{rmk}
\noindent
We say that $\xi_{\eus S}$ is the {\it canonical representative\/} in $\co_\eus S^*\subset\ah^*$. As a by-product of 
Lemma~\ref{lm:simple-main} and our proof of Theorem~\ref{thm:b-orb-dual}, we obtain the following 
description of the tangent space of $\co_\eus S^*$ at $\xi_{\eus S}$.

\begin{prop}  \label{prop:tangent-sp-s}
For $\eus S\in\sah$, the tangent space $\be{\cdot}\xi_{\eus S}\subset\ah^*$ is $T$-stable and the corresponding set of weights (negative 
roots) is $-(\eus S\cup \eus M_{\eus S}^*)$, where $\eus M_{\eus S}^*=(\eus S-\Delta^+)\cap \dah$. 
More precisely,
$-\eus S$ is the set of roots of\/ $\te{\cdot}\xi_{\eus S}$ and $-\eus M_{\eus S}^*$ is the set of roots of\/ 
$\ut{\cdot}\xi_{\eus S}$.
In particular, $\dim\co_\eus S^*= \#(\eus S)+ \#(\eus M_{\eus S}^*)$.
\end{prop}

{\bf Warning}. To describe a tangent space of a $B$-orbit in $\ah$, we use the set 
$\eus M_\eus S=(\eus S+\Delta^+)\cap\dah$, which is the same as $(\eus S+\Delta^+)\cap\Delta^+$, 
because $\ah$ is an ideal. However,
$\eus M_\eus S^*=(\eus S-\Delta^+)\cap \dah$ is usually a proper subset of 
$(\eus S-\Delta^+)\cap \Delta^+$.

Next, we describe the strongly orthogonal set corresponding to the dense $B$-orbit in $\ah^*$.
The {\it upper-canonical set\/} $\eus C^{u}\subset\dah$ is defined inductively, as follows.
We begin with  $\Gamma^*_1=\max(\dah)$, which incidentally is just $\{\theta\}$, and put
$\eus M_1^*=(\Gamma^*_1-\Delta^+)\cap \dah$.
When $\Gamma^*_i$ and $\eus M^*_i$ are already constructed, we define
$\Gamma^*_{i+1}=\max\bigl(\dah\setminus(\bigcup_{j=1}^i(\Gamma^*_j\cup\eus M^*_j)\bigr)$ and
$\eus M^*_{i+1}=(\Gamma^*_{i+1}-\Delta^+)\cap \dah$.
Eventually, we obtain $\Gamma^*_n=\varnothing$ and set 
$\eus C^{u}=\Gamma^*_{1}\cup \Gamma^*_{2}\cup\dots\cup \Gamma^*_{n-1}$. It is quite clear that
$\eus C^{u}$ is strongly orthogonal. Whenever we wish to stress that $\eus C^{u}$ is determined by
$\ah$, we write $\eus C^{u}_\ah$ for it.

\begin{lm}    \label{lm:dense-ah-d}
The upper-canonical set $\eus C^{u}\in\sah$ gives rise to the dense $B$-orbit in $\ah^*$.
\end{lm}
\begin{proof} This is similar to the proof of Lemma~\ref{lm:dense-ah}. It follows from the construction 
of $\eus C^{u}$ that for any $\mu\in \dah\setminus \eus C^{u}$ there exists 
$\gamma_\mu\in\eus C^{u}$ such that $\gamma_\mu-\mu\in\Delta^+$. Furthermore, all the roots
$\gamma_\mu-\mu$ are different in view of Lemma~\ref{lm:simple-main}(ii). Therefore, 
$[\be,\xi_{\eus C^{u}}]=\ah^*$.
\end{proof}

\begin{rmk}   \label{rmk:cascade}
Our procedure of constructing the upper-canonical set in $\dah$ applies perfectly well to arbitrary subsets 
$I$ of $\Delta^+$. But the resulting `canonical' set $\eus C^u_I$ may not be strongly orthogonal. (For 
instance, because the sum of two roots in $\max(I)$ can be a root.) 
However, for $I=\Delta^+$, the procedure does provide a strongly orthogonal set, 
see \cite[Sect.\,2]{jos77}, \cite{ko12}. 
We call it {\it Kostant's cascade} (of strongly orthogonal roots) {\it in\/} $\Delta^+$ and set  
$\eus K=\eus C^u_{\Delta^+}$. If 
$\eus K=\{\gamma_1,\gamma_2,\dots,\gamma_t\}$ is {Kostant's cascade}, then 
$\xi_{\eus K}=\sum_i \xi_{-\gamma_i}\in\ut^*$ is a representative of the dense $B$-orbit in $\ut^*$.
\\ \indent
Furthermore, if $\ce$ is an {\sl ad}-nilpotent ideal of $\be$, then our construction shows 
that $\eus C^{u}_\ce=\eus K\cap \dce$. Hence $\eus C^{u}_\ce$ is strongly orthogonal for any $\ce$. Another good case, which we need below, is that of an arbitrary subset $I\subset \dah$. Here the 
upper-canonical set in $I$ is always strongly orthogonal, since the sum of two roots in $\dah$ is never a root.
\end{rmk}

\begin{ex}     \label{ex:sln-ah-s}
For $\g=\sln$, Kostant's cascade is 
$\eus K=\{\esi_1-\esi_n,\esi_2-\esi_{n-1},\dots,\esi_{[n/2]}-\esi_{n+1-[n/2]}\}$. It consists of the positive 
roots along the antidiagonal. We continue to consider the abelian ideal of shape $(3,3,1)$ in  
$\be(\mathfrak{sl}_n)$, $n\ge 6$, cf. Example~\ref{ex:sln-ah}.
Here $\eus C^{u}=\eus K\cap\dah=\{\esi_1-\esi_{n},\esi_2-\esi_{n-1}\}$ and it is depicted by the diagram
\ $\Yvcentermath1 \tiny
\young(\hfil\hfil\mybullet,\hfil\mybullet\hfil,::\hfil)$ . 
Comparing two canonical sets shows that it may 
happen that $\#\eus C^{l}\ne \#\eus C^{u}$. Furthermore, different abelian ideals may have the 
same upper-canonical set, whereas this is not the case for the lower-canonical sets. Indeed, 
$\eus C^{l}$ contains $\min(\dah)$ and any ideal is completely determined by its minimal elements.
\end{ex}
Our next goal is to describe the Pyasetskii duality for $\ah/B$ and $\ah^*/B$ in terms of $\sah$.
There are the bijections
\begin{gather*}
  \sah\stackrel{Thm.\,\ref{thm:B-orbits-ah}}{\longleftrightarrow}\ah/B
  \stackrel{Pyasetskii}{\longleftrightarrow}\ah^*/B \stackrel{Thm.\,\ref{thm:b-orb-dual}}{\longleftrightarrow}\sah , \\
    (\eus S\in\sah) \mapsto (\co_{\eus S}\in\ah/B) \mapsto 
    \bigl( (\co_{\eus S})^\vee=:\co^*_{\eus S^\vee}\in\ah^*/B\bigr) \mapsto (\eus S^\vee\in\sah) ,
\end{gather*}
and the question is: what is $\eus S^\vee$ in terms of $\eus S$?
We already know the answer in the two extreme cases:
\\
\textbullet \quad If $\eus S=\varnothing$, then $\co_\varnothing=\{0\}\in \ah$ and  
$(\co_{\varnothing})^\vee$ is the dense $B$-orbit in $\ah^*$. Hence $\varnothing^\vee=\eus C^{u}$;
\\
\textbullet \quad Likewise, for $\eus S=\eus C^{l}$ and the dense $B$-orbit in $\ah$, we get $(\eus C^{l})^\vee=\varnothing$.

To discuss the situation for an arbitrary $\eus S\in\sah$, we recall that
the tangent space $[\be,e_{\eus S}]\subset\ah$ is $T$-stable and
the corresponding set of roots is $\eus S\sqcup \eus M_{\eus S}$, where 
$\eus M_{\eus S}=(\eus S+\Delta^+)\cap \Delta^+$ is the set of roots of $[\ut,e_{\eus S}]$, see 
Proposition~\ref{prop:tangent-sp}. We set $J_\eus S=\dah\setminus (\eus S\sqcup \eus M_{\eus S})$.
The following preparatory assertion is required in the proof of Theorem~\ref{thm:Pyasetskii-general} below.
\begin{lm}   \label{lm:preservation}
Suppose that $\gamma^*\in \max(J_\eus S)$ and $\gamma^*-\delta\in J_\eus S$ for some 
$\delta\in\Delta^+$.
\par (1) \ If $\mu\in J_\eus S$ and  $\mu-\delta\in\dah$, then actually 
$\mu-\delta\in J_\eus S$;
\par (2) \ Moreover, if $\mu-2\delta\in\dah$, then both 
$\mu-\delta$ and $\mu-2\delta$ belong to $\in J_\eus S$.
\end{lm}
\begin{proof}
(1) \ Since $\gamma^*,\mu,\gamma^*-\delta$, and $\mu-\delta$ belong to $\dah$,  it follows from Lemma~\ref{lm:simple-main}(ii) that $\gamma^*$ and $\mu$ are not strongly orthogonal. Hence
$\gamma^*-\mu\in\Delta^+$, because $\gamma^*$ is a maximal element of $J_\eus S$. Assume that
$\mu-\delta\in \eus S\sqcup \eus M_{\eus S}$.

{\color{my_color}\textbullet} \ If $\mu-\delta=\gamma\in \eus S$, then $\mu=\gamma+\delta\in
\eus M_{\eus S}$. A contradiction!

{\color{my_color}\textbullet} \ If $\mu-\delta\in \eus M_{\eus S}$, then 
$\mu-\delta=\gamma+\nu$ for some $\gamma\in\eus S$ and $\nu\in\Delta^+$.  By the 
preceding argument, the  roots $\gamma^*, \gamma^*-\delta,\mu,\mu-\delta$ belong to the 
(abelian) ideal $\{\beta\in\Delta^+\mid \beta\curge \gamma\}\subset \dah$. Moreover, since
$\gamma^*, \gamma^*-\delta,\mu\in J_\eus S$, these three roots are orthogonal to $\gamma$.
Consequently, $(\gamma,\mu-\delta)=(\gamma,\gamma+\nu)=0$, too. But the last relation can 
only be satisfied if $\gamma$ is short, $\nu$ is long, and $\|\nu\|^2/\|\gamma\|^2=2$. (This already 
completes the proof in the {\bf ADE}-case!) In general, we note that $\gamma+\nu$ is also short. Thus, 
we have found two short roots such that their difference is a root. Since $\|\nu\|^2/\|\gamma\|^2=2$, the 
sum $\gamma+(\gamma+\nu)$ is also a root. But this contradicts the fact that $\dah$ is abelian.
\\ \indent
All these contradictions prove that $\mu-\delta\in J_\eus S$.

(2) If $\mu-2\delta\in\dah$, then $\mu-\delta\in\dah$ as well, and we conclude that 
$\mu-\delta\in J_\eus S$ in view of part (1). Note that in such a situation, $\mu$ and $\mu-2\delta$ are 
both long and $\delta$ is a short root.
Assume that $\mu-2\delta\in \eus S\sqcup \eus M_{\eus S}$.

{\color{my_color}\textbullet} \ If $\mu-2\delta=\gamma\in \eus S$, then $\mu-\delta=\gamma+\delta\in
\eus M_{\eus S}$. A contradiction!

{\color{my_color}\textbullet} \ If $\mu-2\delta\in \eus M_{\eus S}$, then 
$\mu-2\delta=\gamma+\nu$ for some $\gamma\in\eus S$ and $\nu\in\Delta^+$.
Arguing as above, we obtain that $(\gamma,\gamma+\nu)=0$ and hence $\gamma+\nu$ is short.
On the other hand, we already noticed that $\mu-2\delta$ is long.
\\ \indent
All these contradictions prove that $\mu-2\delta\in J_\eus S$.
\end{proof}
\noindent 
Recall that our construction of the upper-canonical set in $\dah$
applies to {\bf any} subset of $\dah$ and yields an element of $\sah$.

\begin{thm}   \label{thm:Pyasetskii-general}
For any $\eus S\in\sah$ and $\co_\eus S\in\ah/B$, the orbit $(\co_{\eus S})^\vee\in \ah^*/B$ is determined by the
upper-canonical set in $J_\eus S=\dah\setminus (\eus S\sqcup \eus M_{\eus S})$. That is,
$\eus S^\vee$ is the upper-canonical set in $J_\eus S$.
\end{thm}
\begin{proof}
By definition, the Pyasetskii dual orbit for $\co_{\eus S}$ is the dense $B$-orbit in 
$B{\cdot}([\be,e_{\eus S}]^\perp)$. The weights of $V_\eus S:=[\be,e_{\eus S}]^\perp\subset \ah^*$ are 
exactly the negative roots \ $-(\dah\setminus (\eus S\sqcup \eus M_{\eus S}))=-J_\eus S$. 
Let $\eus S^*$ be the upper-canonical set in $J_\eus S$. Then $\eus S^*\in\sah$ and 
$\xi_{\eus S^*}\in V_\eus S$. We claim that $\xi_{\eus S^*}$ belongs to the dense $B$-orbit in 
$B{\cdot}V_\eus S$ and thereby $\eus S^*=\eus S^\vee$. To this end, we show that the reduction 
procedure for elements of $\ah^*$ explained in the proof of Theorem~\ref{thm:b-orb-dual} works also 
for the subspaces of $\ah^*$ of the form $V_\eus S$.

Let $\eta=\sum_{\gamma\in J_\eus S}c_\gamma\xi_{-\gamma}$ be a generic point of $V_\eus S$.
We may assume that $\text{supp}(\eta)=J_\eus S$.
Imitating the general reduction procedure, we set $\Gamma_1^*=\max(J_\eus S)$ and $\eus M^*_1=
(\Gamma_1^*-\Delta^+)\cap J_\eus S$ and write accordingly
\[
   \eta=\sum_{\gamma\in \Gamma_1^*}c_\gamma\xi_{-\gamma}+
   \sum_{\nu\in \eus M^*_1}c_\nu\xi_{-\nu}+\tilde\eta ,
\]
where $\tilde\eta\in\ah^*$ represents the sum related to the roots in 
$J_\eus S\setminus (\Gamma_1^*\sqcup \eus M^*_1)$. Using the disjoint subsets
$\Delta^{(-)}_{\gamma, J_\eus S}=\{\delta\in \Delta^+\mid \gamma-\delta\in J_\eus S\}$, $\gamma\in \Gamma_1^*$, and the corresponding root subgroups of $U$,
we can consecutively kill all the summands in the $\eus M_1^*$-group, without changing the first
group. That is, there is $u\in U$ such that 
\beq    \label{eq:first-step}
    u{\cdot}\eta=\sum_{\gamma\in \Gamma_1^*}c_\gamma\xi_{-\gamma}+\eta' .
\eeq
The problem is that, {\sl a priori}, it might have happened that $\eta'$ does not belong to $V_\eus S$,
that is, $\eta'$ might contain a summand corresponding to a root outside $J_\eus S$. Fortunately,
Lemma~\ref{lm:preservation} guarantee us that this cannot occur. Indeed, if $\tilde\gamma\in\Gamma_1^*$ and $\tilde\gamma-\delta\in J_\eus S$, then using a suitable $\tilde u\in U_\delta$ we may kill
the summand with $\xi_{-\tilde\gamma+\delta}$.  (Note that $\tilde\gamma-\delta\in \eus M^*_1$.)
Suppose that $U_\delta$ also affects a weight vector $\xi_{-\mu}\in V_\eus S$. Then 
$\mu\not\in \Gamma^*_1$ and 
\[
  \tilde u{\cdot}\xi_{-\mu}=\xi_{-\mu}+l_1\xi_{-\mu+\delta}+\dots + l_k\xi_{-\mu+k\delta} ,
\]
where $l_1,\dots,l_k\in\bbk$ and  $\mu-\delta, \dots, \mu-k\delta\in \dah$. Note that $k=1$ in the 
{\bf ADE}-case and $k\le 2$ in the {\bf BCF}-case. (We skip the obvious case of $\GR{G}{2}$, see
Example~\ref{ex:G2}.) 
By Lemma~\ref{lm:preservation}, we have $\mu-\delta, \mu-2\delta\in J_\eus S$.
Hence $\tilde u{\cdot}\xi_{-\mu}\in V_\eus S$ and, in fact, $\tilde u{\cdot}\eta\in V_\eus S$. Iterating 
these elementary simplifications, we conclude that the first reduction step yields a vector 
$u{\cdot}\eta$ in Eq.~\eqref{eq:first-step} such that $\eta'\in V_\eus S$ and $\text{supp}(\eta')\subset
J_\eus S\setminus (\Gamma_1^*\sqcup \eus M^*_1)$.

For generic $\eta$, one may further assume that 
$\text{supp}(\eta')=J_\eus S\setminus (\Gamma_1^*\sqcup \eus M^*_1)$. We then continue our 
reduction procedure with $J_\eus S\setminus (\Gamma_1^*\sqcup \eus M^*_1)$ in place of $J_\eus S$. 
One readily sees that an analogue of Lemma~\ref{lm:preservation} holds for this smaller set of roots. 
Therefore, we stay within $V_\eus S$ during all the subsequent reduction steps. Finally, we obtain that the generic $B$-orbit meeting $V_\eus S$ contains $\xi_{\eus S^*}$, and hence $\eus S^*=\eus S^\vee$.
\end{proof}

\begin{ex}     \label{ex:sln-duality}
In our running example with $\ah$ of shape $(3,3,1)$, take $\eus S=\{\esi_1-\esi_{n-2},\esi_2-\esi_{n}\}$.
Then $\eus S\sqcup \eus M_{\eus S}$ is represented by the picture \ 
$\Yvcentermath1  \tiny\young(\mybullet\ast\ast,\hfil\hfil\mybullet,::\hfil)$ ; and therefore $\eus S^\vee$ is
depicted by the diagram \ 
$\Yvcentermath1  \tiny\young(\hfil\hfil\hfill,\hfil\mybullet\hfil,::\mybullet)$ , i.e., $\eus S^\vee=
\{\esi_2-\esi_{n-1},\esi_3-\esi_{n}\}$.
\end{ex}

\section{Algebras of $U$-invariants and a dimension estimate}
\label{sect:U-inv}

\noindent
In this section, we determine the structure of the invariant algebras $\bbk[\ah]^U$ and $\bbk[\ah^*]^U$. 
Since $U=(B,B)$, these are also the algebras of $B$-semi-invariants.
The assertion that these two algebras are polynomial readily follows 
from \cite{sk77} and the fact that $B$ has dense orbits in $\ah$ and $\ah^*$, respectively.
But in order to determine their Krull dimensions, we invoke our description of 
canonical representatives in the dense $B$-orbits via $\eus C^{l}$ and $\eus C^{u}$, respectively. More
generally, the similar result is valid for the algebras $\bbk[\ce^*]^U$ in place of $\bbk[\ah^*]^U$, i.e.,
for arbitrary {\sl ad}-nilpotent ideals $\ce$.

\begin{lm}
\label{lm:dense-solvable-orb}
Let $\tilde B\to GL(V)$ be a representation of a connected solvable algebraic group $\tilde B$. If $V$ has 
a dense $\tilde B$-orbit, then $\bbk[V]^{(\tilde B,\tilde B)}$ is a polynomial algebra (i.e., $V\md (\tilde B,\tilde B)$ is an affine space) and $\dim V\md (\tilde B,\tilde B)$ equals the number of divisors
in the complement of the dense $\tilde B$-orbit.
\end{lm}
\begin{proof}
This is  a particular case of a general result of Sato and Kimura on prehomogeneous
vector spaces, see \cite[\S\,4, Prop.\,5]{sk77}.
\end{proof}

\begin{lm}   \label{lm:dense-ce-star}
If\/ $\ce\subset\ut$ is an arbitrary {\sl ad}-nilpotent ideal of\/ $\be$, then $B$ has a dense orbit in $\ce^*$. 
\end{lm}
\begin{proof}
It is well known that $B$ has a dense orbit in $\ut^*$, see Remark~\ref{rmk:cascade}. 
If $\eta\in\ut^*$ is a representative of the dense $B$-orbit and $\ce\subset \ut$ is $B$-stable, then
$\bar\eta=\eta\vert_{\ce}\ne 0$. Since $\be{\cdot}\eta=\ut^*$ and the surjection $\ut^*\to\ce^*$ is
$B$-equivariant, we obtain $\be{\cdot}\bar\eta=\ce^*$, i.e., 
the $B$-orbit of $\bar\eta$ is dense in $\ce^*$.
\end{proof}

\begin{thm}   \label{thm-U-inv-ah}
For any abelian ideal $\ah$, we have $\ah\md U\simeq \mathbb A^p$, where
$p=\#(\eus C^{l}_\ah)$. 
\end{thm}
\begin{proof}
Since $\ah$ contains a dense $B$-orbit, $\bbk[\ah]^U$ is 
polynomial in view of  Lemma~\ref{lm:dense-solvable-orb}. If 
$\eus C^{l}=\eus C^{l}_\ah=\{\gamma_1,\dots,\gamma_p\}$ and 
$V=\bigoplus_{i=1}^p \bbk e_{\gamma_i}$, then it follows from Theorem~\ref{thm:B-orbits-ah} that
$ \co_{ \eus C^{l}}\cap V=\{\sum_{i=1}^p a_i e_{\gamma_i}\mid a_1\cdots a_p\ne 0\}$.
Furthermore, different elements of this intersection belong to different $U$-orbits and all these $U$-orbits 
are isomorphic. That is, the dense $B$-orbit splits into an $p$-parameter family of isomorphic $U$-orbits, 
which implies that generic $U$-orbits in $\ah$ are of codimension $p$~\cite[Ch.\,1, n.\,2, Prop.\,2]{brion}.
Therefore,  $\trdeg \bbk(\ah)^U=p$~\cite[Ch.\,1, n.\,6]{brion}. Finally, since $U$ has no non-trivial characters and $\ah$ is factorial, $\bbk(\ah)^U$ is the quotient field of $\bbk[\ah]^U$.
\end{proof}

\begin{thm}    \label{thm-U-inv-ahs}
For any {\sl ad}-nilpotent ideal $\ce\subset\ut$, we have $\ce^*\md U\simeq \mathbb A^m$, where
$m=\#(\eus C^{u}_\ce)$. 
\end{thm}
\begin{proof}
Using Lemmas~\ref{lm:dense-solvable-orb} and \ref{lm:dense-ce-star}, we conclude that 
$\bbk[\ce^*]^U$ is polynomial, i.e., $\ce^*\md U$ is an affine space. Next, we use the fact
that, for $\eus C^{u}_\ce=\eus K\cap \Delta_\ce=\{\tilde\gamma_1,\dots,\tilde\gamma_m\}$, 
$\xi=\sum_{i=1}^m\xi_{-\tilde\gamma_i}$ lies in the dense $B$-orbit $\co^*\subset\ce$. Then one easily 
verifies that $\co^*$ also contains an $m$-parameter family of $U$-orbits of codimension $m$.
\end{proof}

\begin{rmk}
If $\ce$ is an arbitrary {\sl ad}-nilpotent ideal, then {\it (1)} $B$ may not have a dense orbit in $\ce$ 
and {\it (2)} the algebra $\bbk[\ce]^U$ can fail to be polynomial. It is not even clear that $\bbk[\ce]^U$ 
is always finitely generated! 
\end{rmk}

Applying the last assertion of Lemma~\ref{lm:dense-solvable-orb} to the abelian ideals, we  obtain 

\begin{cl}   \label{cor:orb-codim1}
The number of $B$-orbits of codimension 1 in $\ah$ (resp. $\ah^*$) equals
$\#\eus C^{l}$ (resp. $\#\eus C^{u}$). 
\end{cl}

\begin{rmk}   \label{rem:ochen-raznye}
For the abelian ideals, both algebras $\bbk[\ah]^U$ and $\bbk[\ah^*]^U$ are polynomial, but they 
are rather different. Let us write $\bbk[\ah]^U=\bbk[f_1,\dots,f_p]$ and 
$\bbk[\ah^*]^U=\bbk[h_1,\dots,h_m]$, where $\{f_i\}, \{h_j\}$ are two sets of
algebraically independent semi-invariants of $B$.
\\ \indent
1) 
Examples~\ref{ex:sln-ah} and~\ref{ex:sln-ah-s} show that it can happen that $p\ne m$, i.e., 
the Krull dimensions of two algebras are different.
\\  \indent
2) Since $\bbk[\ah^*]^U=\eus S(\ah)^U\subset \eus S(\g)^U$, the basic semi-invariants 
$h_1,\dots,h_m$ have dominant $T$-weights with respect to $B$. While the $T$-weights of 
$f_1,\dots,f_p$ belong to the cone generated by $-\dah$.
\\  \indent
3) It is easily seen that the number of generators (semi-invariants) of degree 1 equals: \\ 
\centerline{$\#\min(\dah)$ for $\bbk[\ah]^U$, \ and \ $\#\max(\dah)=1$  for $\bbk[\ah^*]^U$.}
\end{rmk}
\begin{ex}   \label{ex:U-inv}  
In our eternal example with $\ah$ of shape $(3,3,1)$, let $e_{(i,j)}\in\ah$ (resp. $\xi_{-(i,j)}\in\ah^*$)
denote the weight vector corresponding to $\gamma=\esi_i-\esi_j$ (resp. $-\gamma$). We regard 
$e_{(i,j)}$ and $\xi_{-(i,j)}$ as linear functions on $\ah^*$ and $\ah$, respectively. Then a direct verification shows that 

{\color{my_color}\textbullet}\quad $\bbk[\ah^*]^U$ is freely generated by $e_{(1,n)}$ and 
$\begin{vmatrix} e_{(1,n-1)} & e_{(1,n)} \\ e_{(2,n-1)} & e_{(2,n)} 
\end{vmatrix}$;

{\color{my_color}\textbullet}\quad $\bbk[\ah]^U$ is freely generated by $\xi_{-(2,n-2)}$, $\xi_{-(3,n)}$,  and
$\begin{vmatrix} \xi_{-(1,n-2)} & \xi_{-(1,n-1)} \\ \xi_{-(2,n-2)} & \xi_{-(2,n-1)} 
\end{vmatrix}$.
\end{ex}
\noindent
Let $\z_\be(\ce)$ (resp. $Z_B(\ce)$) denote the centraliser of $\ce$ in $\be$ (resp. $B$).
For an abelian ideal $\ah$, we have $\z_\be(\ah)\supset \ah$. Therefore, the $B$-action on $\ah$ has 
the large {\it ineffective kernel\/} $Z_B(\ah)$. Since $B$ has an open orbit in $\ah$, this implies that 
$\dim\be\ge \dim\z_\be(\ah)+\dim\ah\ge 2\dim\ah$. 
It is known that $\z_\be(\ah)=\ah$ if and only if $\ah$ is maximal~\cite{JEMS}. Therefore, if 
$\dim\be=2\dim\ah$, then $\ah$ is maximal. 
This equality occurs only for the unique maximal abelian ideal in
$\be(\spn)$, see \ref{subs:C} below. However, there is a more precise inequality. 
Recall that the {\it index\/} of a Lie algebra $\q=\Lie(Q)$, denoted $\ind\q$, is the minimal codimension of $Q$-orbits in $\q^*$.

\begin{prop}   \label{prop:dim-estimate} 
For any abelian ideal $\ah$, we have $2\dim\ah\le \dim\ut+\#(\eus K)=\dim\be-\ind\be$. Furthermore, if the equality holds, then $\eus K\subset \dah$.
\end{prop}
\begin{proof} If $\eus K=\{\gamma_1,\dots,\gamma_t\}\subset\Delta^+$, then
$\xi_\eus K=\sum_{i=1}^t \xi_{-\gamma_i}\in\ut^*$ belongs to the open $B$-orbit in $\ut^*$.
Here $\be^{\xi_\eus K}=\te^{\xi_\eus K}\oplus\bigl(\bigoplus_{i=1}^t \ut_{\gamma_i}\bigr)$, where
$\te^{\xi_\eus K}=\{t\in \te\mid \gamma_i(t)=0, \ i=1,\dots,t\}$.  
If $p:\ut^*\to \ah^*$ is the natural surjection, then
$\bar\xi:=p({\xi_\eus K})$ lies in the open $B$-orbit in $\ah^*$ and $\be^{\bar\xi}\supset 
\te^{\xi_\eus K} \oplus \ah$, since $\ah$ is abelian. Hence
\[
   \dim\ah=\dim\be-\dim\be^{\bar\xi}\le 
   \dim\be-\dim\te+\#(\eus K)-\dim\ah . 
\]
Hence $2\dim\ah\le \dim\ut+\#(\eus K)$.
It is also well known that $\ind\be=\dim\te-\#(\eus K)=\dim\te^{\xi_\eus K}$.
If $\eus K\not\subset\dah$, then $\bar\xi$ has fewer weight summands than $\xi_\eus K$. Therefore, 
$\dim\te^{\bar\xi}> \dim\te^{\xi_\eus K}$ and the displayed inequality appears to be strict.
\end{proof}

\begin{ex}   \label{ex:ravensto-2cases}
It is easily verified that $\eus K\subset \dah$ only in the following two cases: 

{\color{my_color}\textbullet} \quad $\g=\sln$ and $\ah$ is an abelian ideal of maximal dimension, which is $[n^2/4]$.
(There are two such ideals if $n$ is odd, and one ideal if $n$ is even.)
Here $\#(\eus K)=[n/2]$, $\ind\be=[\frac{n-1}{2}]$, and $\dim\be=\frac{n(n+1)}{2}-1$. 

{\color{my_color}\textbullet} \quad $\g=\spn$ and $\ah$ is the unique maximal abelian ideal. Here
$\ind\be=0$, $\dim\be=n^2+n$, and $\dim\ah=(n^2+n)/2$.
\\
Thus, in both cases one obtains the equality $\dim\be-\ind\be= 2\dim\ah$.
\end{ex}

\section{Counting $B$-orbits in the abelian nilradicals} 
\label{sect:ab-nilrad}

\noindent
Let $\Pi=\{\ap_1,\dots,\ap_n\}$ be the set of simple roots in $\Delta^+$. 
Let $P=L{\cdot}P^u$ be a standard parabolic subgroup of $G$, where $L$ is the standard Levi subgroup 
(i.e., $L\supset T$) and $P^u$ is the unipotent radical of $P$. If $\Delta_L\subset \Delta$ is the set of 
roots of $\el=\Lie(L)$, then $\Pi_L=\Delta_L\cap \Pi$ is a set of simple roots for $L$; furthermore, 
$P$ is maximal if and only if $\Pi_L=\Pi\setminus \{\ap_i\}$ for some $i$. Whenever 
we wish to stress that  $P$ is maximal and determined by $\ap_i$, we write $P=P_i$ for it.
The group $P^u$ is abelian if and only if $P=P_i$ and the coefficient of $\ap_i$ in the 
expression $\theta=\sum_{i=1}^n k_i\ap_i$ is equal to $1$.
Then $\p^u=\Lie(P^u)$ is a {\sl maximal\/} abelian ideal (but not all maximal abelian ideals are of this 
form!). The relevant simple roots, with numbering from \cite[Table\,1]{t41}, are presented below:

\begin{center}
\begin{tabular}{|cc||cc|}
\hline
$\GR{A}{n}$ & $\ap_1,\dots,\ap_n$ & $\GR{D}{n}, \ n\ge 3$ & $\ap_1,\ap_{n-1},\ap_n$ \\ \hline
$\GR{B}{n}, \ n\ge 2$ & $\ap_1$ & $\GR{E}{6}$ & $\ap_1,\ap_5$ \\ \hline
$\GR{C}{n}, \ n\ge 2$ & $\ap_n$ & $\GR{E}{7}$ & $\ap_1$  \\
\hline
\end{tabular}
\end{center}
\vskip1.5ex
\begin{rmk}     \label{rem:max-ab} 
It was observed empirically in \cite{pr01} that the number of maximal abelian $\be$-ideals equals the number of long roots in $\Pi$ (in the simply laced case, all simple roots are assumed to be long). A uniform explanation of this phenomenon is given in \cite[Cor.\,3.8]{imrn}. Therefore, series $\GR{A}{n}$ provides the only case in which all maximal abelian ideals are \textsf{ANR}s.
\end{rmk}
In this section, we compute the number of $B$-orbits in all abelian nilradicals (\textsf{ANR}) $\p^u$.
Moreover, we determine the statistic on $\p^u/B$ that associates the number $\#\eus S$ to the orbit
$\co_\eus S$. 
\par 
In general, let $(\ah/B)_i$ denote the set of all $B$-orbits $\co_\eus S\subset\ah$ such that 
$\#\eus S=i$. Clearly, $\#(\ah/B)_0=1$ and $\#(\ah/B)_1=\dim\ah$ for any abelian 
ideal $\ah$.

For an \textsf{ANR} $\p^u$, the ineffective kernel $Z_B(\p^u)$ equals $P^u$ and 
$B/P^u\simeq B_L:=B\cap L$. Therefore the $B$-orbits in $\p^u$ coincide with the $B_L $-orbits.
In the context of $B_L$-orbits, one can also think of $(\p^u)^*$ as $\p^u_-$, the opposite nilradical.
The Weyl group of $L$, $W_L $, acts transitively on the set of roots of the same length in $\Delta_{\p^u}$ 
\cite[Lemma\,2.6]{RRS}. If $\p=\p_i$ and $w_{0,L}\in W_L $ is the longest element, then
$w_{0,L}(\theta)=\ap_i$. 
Since $L$ is reductive and both $\p^u$ and $(\p^u)^*$ are $L$-modules, the posets
$\p^u/B=\p^u/B_L$ and $(\p^u)^*/B=(\p^u)^*/B_L$ are isomorphic. But one can say more!
\begin{prop}   \label{prop:isom-p-u}
Using our parametrisation of the $B$-orbits via $\mathfrak S_{\p^u}$,
the natural poset isomorphism $\p^u/B_L \simeq (\p^u)^*/B_L$ is given by the action of $w_{0,L}$ on 
$\Delta_{\p^u}$ and hence on $\mathfrak S_{\p^u}$.
\end{prop}
\begin{proof}
Let $\vartheta$ be the Weyl involution of $L$ associated with $(B_L,T)$. That is, $\vartheta(B_L)\cap B_L=T$ and $\vartheta(t)=t^{-1}$ for any $t\in T$. Set $B_L^-=\vartheta(B_L)$. As is well known, any finite-dimensional representation of $L$ twisted with $\vartheta$ is equivalent to the dual one.
Therefore, the closure relation for the $B_L$-orbits in $(\p^u)^*$ corresponds to the closure relation for
$B_L^-$-orbits in $\p^u$, that is, $(\p^u)^*/B_L\simeq \p^u/B_L^-$. It is also clear that
$\p^u/B_L^-\simeq \p^u/B_L$. In terms of our canonical representatives of $B$-orbits and the set
$\mathfrak S_{\p^u}$, the isomorphisms 
\[
   (\p^u)^*/B_L\isom \p^u/B_L^-\isom \p^u/B_L
\]
are described as follows. Suppose that 
$\eus S=\{\gamma_1,\dots,\gamma_k\}\subset \mathfrak S_{\p^u}$. Then
\[
  \co^*_\eus S=B_L{\cdot}(\sum_{i=1}^k \xi_{-\gamma_i})\mapsto 
  B_L^-{\cdot}(\sum_{i=1}^k e_{\gamma_i})\mapsto
  B_L{\cdot}(\sum_{i=1}^k e_{w_{0,L}(\gamma_i)})=\co_{w_{0,L}(\eus S)} \ ,
\]
and this is exactly what we need. (Here we use the fact that $w_{0,L}$ takes $B_L^-$ to $B_L$.)
\end{proof}
It follows from this proposition that $w_{0,L}(\eus C^u)=\eus C^{l}$, which can also be proved directly.
See also \cite[Sect.\,1]{aura} for other properties of $\eus C^{l}$ and $\eus C^{u}$, where these are called "canonical strings" of roots.
It is known that $G/L$ is a symmetric variety of Hermitian type, and $\#(\eus C^{l})=\#(\eus C^{u})$
equals the rank of $G/L$, denoted $\rk(G/L)$. It is also true that $\eus C^{l}$ (and $\eus C^{u}$) consists
of long roots (if there are two root lengths) and these are
strongly orthogonal sets in $\Delta_{\p^u}$ of maximal cardinality.

\subsection{$\g=\slN$} 
The nilradicals of the maximal parabolic subalgebras are represented by 
the Young diagrams of rectangular shape $m{\times}(N{-}m)$, $m=1,\dots, N{-}1$, see 
Example~\ref{ex:sln-ah}. Actually, one easily computes the number of $B$-orbits in any abelian ideal of 
a rectangular shape. Let $\ah$ correspond to the rectangle of size $m\times n$, with $m+n\le N$.
A subset $\eus S\subset \dah$ is (strongly) orthogonal if and only if the corresponding roots lie in the 
different rows and different columns of the rectangle. Therefore $(\ah/B)_k=\varnothing$ for $k > \min\{m,n\}$ and 
\[
  \#(\ah/B)_k = k! \genfrac{(}{)}{0pt}{}{m}{k} \genfrac{(}{)}{0pt}{}{n}{k} \quad \text{ for } k=0,1,\dots,\min\{m,n\} .
\]

\subsection{$(\sono,\ap_1)$}  \label{subs:B}
Here $\dim (\p_1)^u=2n-1$, $\Delta_{(\p_1)^u}=\{\esi_1-\esi_2,\dots,\esi_1-\esi_n,\esi_1,\esi_1+\esi_n,\dots,
\esi_1+\esi_2\}$, $(L,L)=SO_{2n-1}$,  and the easy answer is:

\begin{center}
\begin{tabular}{c|ccc||c}
$k$ & 0 & 1& 2 & $\Sigma$ \\ \hline
$ \#\bigl((\p_1)^u/B_L \bigr)_k$ &1& $2n-1$ & $n-1$ & $3n-1$\\
\end{tabular} ,
\end{center}
where the last column  indicates the total number of $B_L$-orbits (=\,$B$-orbits) in $(\p_1)^u$.

\subsection{$\g=\sone$}   \label{subs:D} \leavevmode \par
a) \ $(\sone,\ap_1)$.  Here $\dim (\p_1)^u=2n-2$, $\Delta_{(\p_1)^u}=\{\esi_1-\esi_2,\dots,\esi_1-\esi_n,\esi_1+\esi_n,\dots,
\esi_1+\esi_2\}$, $(L,L)=SO_{2n-2}$,  and

\begin{center}
\begin{tabular}{c|ccc||c}
$k$ & 0 & 1& 2 & $\Sigma$ \\ \hline
$ \#\bigl((\p_1)^u/B_L \bigr)_k$ &1& $2n-2$ & $n-1$ & $3n-2$\\
\end{tabular} .
\end{center}

b) \ $(\sone,\ap_{n-1} \text{ or } \ap_n)$.  Here $\Delta_{(\p_n)^u}=\{\esi_i+\esi_j \mid 1\le i< j\le n\}$, 
$L=GL_n$, and the $GL_n$-module $(\p_n)^u$ is isomorphic to the space of skew-symmetric 
$n$ by $n$ matrices. An orthogonal set of cardinality $k$ is given by roots
$\esi_{i_1}+\esi_{i_2},\dots,\esi_{i_{2k-1}}+\esi_{i_{2k}}$, where all indices are different.
For a $2k$-element set $\{i_1,i_2,\dots,i_{2k}\}$, the number of its partitions into $k$ pairs equals
$\frac{(2k)!}{k!\,2^k}
$. (If $k>0$, this is also 
the number of summands in the pfaffian of a generic skew-symmetric matrix of order $2k$.)
Therefore,
\beq     \label{eq:serD}
  d_{n,k}:=\#\bigl((\p_n)^u/B_L \bigr)_k=\genfrac{(}{)}{0pt}{}{n}{2k}\frac{(2k)!}{k!\,2^k} , \quad k=0,1,\dots,[n/2] .
\eeq
The total number of $B_L$-orbits is given by the following integers:

\begin{center}
\begin{tabular}{c|cccccccc|}
$n$ & 1& 2 & 3&4&5&6&7&\dots \\ \hline
$ \#\bigl((\p_n)^u/B_L \bigr)$ & 1& 2& 4& 10& 26& 76& 232&\dots \\
\end{tabular}
\end{center}
\vskip1.5ex
This is sequence \href{http://oeis.org/A000085}{A000085}  in OEIS~\cite{oeis}. Of course, 
the same numbers occur for $\ap_{n-1}$ and $(\p_{n-1})^u$, i.e., $d_{n,k}=d_{n-1,k}$.

\subsection{$(\spn,\ap_n)$}  \label{subs:C}
Here $(\p_n)^u$ is the unique maximal abelian ideal, 
$\Delta_{(\p_n)^u}=\{\esi_i+\esi_j \ (1\le i< j\le n),\ 2\esi_i \ (i=1,\dots,n) \}$, $L=GL_n$,
and the $GL_n$-module $(\p_n)^u$ is isomorphic to the space of symmetric $n$ by $n$ matrices.

A strongly orthogonal set of cardinality $k$ in $\Delta_{(\p_n)^u}$ is of the form
$\esi_{i_1}{+}\esi_{i_2},\dots,\esi_{i_{2t-1}}{+}\esi_{i_{2t}},$ $2\esi_{j_{1}},\dots,2\esi_{j_{k-t}}$, where 
$0\le t\le k$ and all indices are different. Therefore, the number of such $k$-elements sets,
$ \#\bigl((\p_n)^u/B_L \bigr)_k=c_{n,k}$, equals
\beq     \label{eq:serC}
   c_{n,k}=\sum_{t=0}^k \genfrac{(}{)}{0pt}{}{n-2t}{k-t} d_{n,t} \ ,
\eeq
where $d_{n,t}$ occurs in Eq.~\eqref{eq:serD}. Since  $k-t\le n-2t$, we get also the constraint
$t\le n-k$. That is, the actual range of summation in Eq.~\eqref{eq:serC} is $0\le t\le \min\{k,n-k\}$.
Using this and Eq.~\eqref{eq:serD}, one readily derives the symmetry $c_{n,k}=c_{n,n-k}$.
The total number of $B_L$-orbits in $(\p_n)^u$ is given by the following integers:

\begin{center}
\begin{tabular}{c|ccccccc|}
$n$ & 1& 2 & 3& 4& 5& 6& \dots \\ \hline
$ \#\bigl((\p_n)^u/B_L \bigr)$ & 2& 5 &14 & 43& 142 & 499 &\dots \\
\end{tabular}
\end{center}
\vskip1.5ex
This is sequence \href{http://oeis.org/A005425}{A005425}  in OEIS~\cite{oeis}.

The symmetry for the numbers $c_{n,k}$ suggests that there ought to be a natural one-to-one 
correspondence between the strongly orthogonal sets of cardinality $k$ and $n-k$. Here it is.
Suppose that $\#(\eus S)=k$ and $\eus S=\tilde{\eus M}\cup\eus M$, where $\tilde{\eus M}$ 
(resp. $\eus M$) consists
of short (resp. long) roots; $\#(\tilde{\eus M})=t$ and $\#(\eus M)=k-t$. Here we use $2t$ indices in
$\tilde{\eus M}$ and $k-t$ indices in ${\eus M}$ (all the indices are different!). 
We then associate to $\eus S$ the
set $\eus S'=\tilde{\eus M}\cup\eus M'$, where $\tilde{\eus M}$ is the same as in $\eus S$ and the new set of long roots $\eus M'$ uses all the indices
that do not occur in $\eus S$. Hence $\#(\eus M')=n-k-t$  and $\#(\eus S')=n-k$, as required. 
Curiously, this is the only case in which the sequence $ \#\bigl((\p)^u/B_L \bigr)_k$, with 
$1\le k\le \rk(G/L)$, is symmetric!

\subsection{($\GR{E}{7},\ap_1)$}  Here $(L,L)=\GR{E}{6}$ and $(\p_1)^u$ is  
a simplest (27-dimensional) $\GR{E}{6}$-module. 

Let $(\mu_1,\mu_2)$ be a pair of orthogonal roots in $\Delta_{(\p_1)^u}$. Then 
$(\mu_1,\mu_2)\underset{W(\GR{E}{6})}{\sim} (\ap_1,\mu_2')$. It is not hard to compute that there are 10
roots in $\Delta_{(\p_1)^u}$ that are orthogonal to $\ap_1$. Therefore,
$\#\bigl((\p_1)^u/B_L \bigr)_2=27{\cdot}10/2!=135$.

Let $(\mu_1,\mu_2,\mu_3)$ be a triple of orthogonal roots in $\Delta_{(\p_1)^u}$. Then 
$(\mu_1,\mu_2,\mu_3)\underset{W(\GR{E}{6})}{\sim} (\ap_1,\mu_2',\mu_3')$ and 
the stabiliser of $\ap_1$ in $W(\GR{E}{6})$ is $W(\GR{D}{5})$. The 10 roots that are orthogonal to $\ap_1$
form the weight system of the simplest (10-dimensional) representation of $\GR{D}{5}$. Therefore, for
given $\mu_2'$, there is a unique $\mu_3'$ that is orthogonal to $\mu_2'$. Therefore,
$\#\bigl((\p_1)^u/B_L \bigr)_3=27{\cdot}10{\cdot}1/3!=45$.

Thus, the complete answer is: \quad
\begin{tabular}{c|cccc||c}
$k$ & 0 & 1&2&3 & $\Sigma$ \\ \hline
$ \#\bigl((\p_1)^u/B_L \bigr)_k$ &1& 27 & 135 & 45 & 208 \\
\end{tabular} \ .

\subsection{($\GR{E}{6},\ap_1 \text{ \normalfont or } \ap_5)$}  
Here $(L,L)=\GR{D}{5}$ and $(\p_1)^u$ is isomorphic to a half-spinor
(16-dimensional) $\GR{D}{5}$-module. The argument in this case is similar to the previous one (and shorter!). The answer is:

\begin{center}
\begin{tabular}{c|ccc||c}
$k$ & 0 & 1&2 & $\Sigma$ \\ \hline
$ \#\bigl((\p_1)^u/B_L \bigr)_k$ &1& 16 & 40 & 57 \\
\end{tabular} \ .
\end{center}

\section{Closures of $B$-orbits and involutions in the Weyl group} 
\label{sect:closure}

\noindent
Let  $\BV$ be a $Q$-module with finitely many orbits and let $\ov{\co}$ denote the closure of a  $Q$-orbit $\co$ in $\BV$. One makes $\BV/Q$ a finite poset by 
letting $\co_1\curle\co_2$ if $\co_1\subset \ov{\co_2}$. Write $\co_1\prec\co_2$ if $\co_1\curle\co_2$ 
and $\co_1\ne\co_2$. As usual, we say that $\co_2$ {\it covers\/} $\co_1$, if $\co_1\prec\co_2$ and 
there is no orbits $\co'$ such that $\co_1\prec\co'\prec\co_2$. 
\\ \indent
Below, we consider the case in which $Q=B$ and $\BV$ is either $\ah$ or $\ah^*$. 
Any homogeneous space of a solvable algebraic group is affine. 
Therefore $\ov{\co}\setminus\co$ is a  union of divisors in $\ov{\co}$. Hence in our situation 
\[
   \text{ $\co_2$ {covers} $\co_1$ \ {\sl if and only if} \  $\co_1\curle\co_2$ and $\dim\co_1+1=\dim\co_2$.}
\]

\begin{qtn}      \label{prob:1}
Describe the $B$-orbit closures in $\ah$ and/or in $\ah^*$.
\end{qtn}
\noindent
As both $\ah/B$ and $\ah^*/B$ are parameterised by $\sah$, we seek a description in terms strongly
orthogonal sets of roots. Here the dimension of any 
orbit can be computed using Propositions~\ref{prop:tangent-sp} and \ref{prop:tangent-sp-s}, which already provides a rather good approximation to the structure of the Hasse diagram of both posets.

At this writing, we do not know a general solution for $\ah/B$ or $\ah^*/B$. We only suggest below a 
conjecture for the \textsf{ANR} $\p^u$.
Prior to that, we discuss certain relations between the posets $\ah/B$ and $\ah^*/B$, and some 
related results for classical Lie algebras. 

We have two bijections $\ah/B\longleftrightarrow\ah^*/B$ at our disposal:

1) \ the Pyasetskii duality (or {\it P-duality}) $\co\longleftrightarrow\co^\vee$, which is quite useful and general;

2) \ the dull bijection  $\co_\eus S\longleftrightarrow\co^*_\eus S$, $\eus S\in\sah$. This  relies on the 
special fact that one has two independent classifications of $B$-orbits that exploit the same parameter 
set.

To some extent, the {\sl general\/} P-duality resembles an anti-isomorphism of posets. Set
\begin{gather*}
    (\BV/Q)_{sp}=\{\co\in \BV/Q \mid  \ov{\co} \ \text{ is a subspace of $\BV$}\} \quad
\text{ and}  \\
    (\BV^*/Q)_{sp}=\{\co^*\in \BV^*/Q \mid  \ov{\co^*} \ \text{ is a subspace of $\BV^*$}\} .
\end{gather*}
Note that these sub-posets contain quite a few elements, if $Q$ is solvable. If $\co\in (\BV/Q)_{sp}$ and 
$v\in\co$, then $\q{\cdot}v=\ov{\co}$ and $(\q{\cdot}v)^\perp\in\BV^*$ is also a $Q$-stable subspace.
Therefore $\co^\vee\in (\BV^*/Q)_{sp}$ and the P-duality induces an anti-isomorphism of  $(\BV/Q)_{sp}$ 
and $(\BV^*/Q)_{sp}$, i.e., if $\co_1,\co_2\in (\BV/Q)_{sp}$ and $\co_1\prec\co_2$, then
$\co_2^\vee\prec\co_1^\vee$. But for the
other pairs of $Q$-orbits, the P-duality behaves unpredictably.  
If $\tilde\co_1,\tilde\co_2\in \BV/Q$ and $\tilde\co_1\prec \tilde\co_2$,  then 
it can also happen that $\tilde\co_1^\vee$ and $\tilde\co_2^\vee$ are incomparable or even 
$\tilde\co_1^\vee\prec \tilde\co_2^\vee$.
\\ \indent
But the dull bijection (for $\BV=\ah$) seems to have no useful properties at all. For instance, it can 
happen that $\ov{\co_{\eus S}}$ is a  subspace, but $\ov{\co^*_{\eus S}}$ is not (and vice versa).

Let us record some elementary properties of the closure relation referring to $\sah$.

1) If $\eus S\in\sah$, $\gamma\in\eus S$, and $\eus S'=\eus S\setminus\{\gamma\}$, then 
$\co_{\eus S'}\prec \co_\eus S$. But it is {\bf not} always the case that $\co_\eus S$ covers
$\co_{\eus S'}$.

2) If ${\co_{\eus S}}\in (\ah/B)_{sp}$, then $\ov{\co_{\eus S}}=:\ah'\subset\ah$ is 
a smaller abelian ideal.
Then $\mathfrak S_{\ah'}=\{\eus S'\in \sah \mid \eus S'\subset \Delta_{\ah'}\}$ and 
$\ov{\co_\eus S}=\cup_{\eus S'\in \mathfrak S_{\ah'}}\co_{\eus S'}$. \quad
(Yet, this does not provide a description of the orbits covered by $\co_\eus S$.)
\\  \indent 
Of course, similar assertions 1)--2) are valid for the orbits $\co^*_\eus S\subset\ah^*$.

\subsection{The $\sln$-case}   \label{subs:sln}
The union of spherical nilpotent $SL_n$-orbits in $\g=\sln$ consists of matrices $X$ such that $X^2=0$,
where $X^2$ is the usual matrix square of $X$
\cite[Sect.\,4]{sph-Man}. Therefore $\{X\in\sln\mid X^2=0\}/B$ is finite. The classification of these 
$B$-orbits is obtained in~\cite{rothbach}, cf. also \cite{br}. But earlier efforts has been devoted to 
a smaller $B$-stable subvariety
\[
   \be^{\langle 2\rangle}=\{X\in \be\mid X^2=0\}=\{X\in \ut\mid X^2=0\} .
\]
In \cite{Anna00}, Anna Melnikov proved that $\be^{\langle 2\rangle}/B$ is in a one-to-one 
correspondence with the set of all involutions in the symmetric group $\mathbb S_n=W(\sln)$ and 
pointed out a representative in every $B$-orbit in $\be^{\langle 2\rangle}$. Later on, she described the 
closures of $B$-orbits in $\be^{\langle 2\rangle}$~\cite[Sect.\,3]{Anna06}. \\
A connection with our results stems from the observation that any involution in $\mathbb S_n$ can 
uniquely be written as the product of reflections corresponding to an orthogonal set of positive 
roots (= product of commuting transpositions).  If $\sigma\in\mathbb S_n$ is an involution and 
$\{\gamma_1,\dots,\gamma_k\}$ is the corresponding orthogonal set, then 
$\sum_{i=1}^k e_{\gamma_i}\in\ut$ is actually a representative (described in \cite{Anna00}) of the 
$B$-orbit in $\be^{\langle 2\rangle}$ corresponding to $\sigma$. Since any abelian ideal $\ah$ belongs 
to $\be^{\langle 2\rangle}$ and our canonical representatives for $B$-orbits in $\ah$ coincide with those
obtained by Melnikov, results of \cite{Anna06} yield a description of the $B$-orbits closures in $\ah$.

In~\cite{ign}, Ignatyev considers the finite set of $B$-orbits in $\ut^*$ that is obtained from the above
representatives of $B$-orbits in $\be^{\langle 2\rangle}$ via the ``dull bijection''. If
$\sigma\in\textsf{Inv}(\mathbb S_n)$ and $\{\gamma_1,\dots,\gamma_k\}$ are as above, then he 
considers the $B$-orbit of $\sum_{i=1}^k \xi_{-\gamma_i}\in\ut^*$, which we denote by $\co_\sigma^*$. The resulting family of orbits 
forms a rather odd and artificial conglomerate. For instance, it contains the dense $B$-orbit in $\ut^*$,
but not all $B$-orbits. Anyway, one can look at the closure relation for 
$\{\co^*_\sigma\}_{\sigma\in\textsf{Inv}(\mathbb S_n)}$. The surprising answer is that 
$\co^*_{\sigma'} \subset\ov{\co^*_\sigma}$ if and only if
$\sigma' \le \sigma$ with respect to the Bruhat order~\cite[Theorem\,1.1]{ign}.
Combined with the surjection $\ut^*\to \ah^*$ and our classification of $B$-orbits in $\ah^*$ via $\sah$, 
this yields  a description of $B$-orbit closures in $\ah^*$. 

\subsection{$\g=\spn$ or $\son$}   \label{subs:sp+so}
Bagno--Cherniavsky~\cite{bagno-chern} and Cherniavsky~\cite{chern11}  describe the orbits of 
$B(GL_n)$ in the spaces of symmetric and skew-symmetric $n$ by $n$ matrices, respectively. Their 
classifications are stated in 
terms of "partial permutations" in $\mathbb S_n$. But from our point of view, these are instances of 
abelian nilradicals $\p^u$ associated with $\g=\sone$ and $\spn$, respectively
(cf. \ref{subs:D}(b) and \ref{subs:C}). In both cases, the corresponding Levi subgroup is $GL_n$ and,
as in Section~\ref{subs:sln}, these partial permutations naturally correspond to the strongly orthogonal sets of roots in $\Delta_{\p^u}$.

\begin{rmk}   \label{rem:kritika}
It is claimed in both articles that the closure of $B$-orbits can be described 
via certain ``rank-control matrices'', see \cite[Lemma\,5.2]{bagno-chern} and \cite[Prop.\,4.3]{chern11}), 
which resembles, in fact, the description of Melnikov in \cite{Anna06}. But in place of a solid proof, the 
authors only briefly refer to Theorem~15.31 in~\cite{MS}, where the action of another group on another 
space is considered! In my opinion, unjustified assurances that ``differences can be easily 
overwhelmed" cannot be accepted as a proof. It also remains unclear to me whether the authors 
of~\cite{bagno-chern,chern11} realise that their Borel subgroups are different from that in \cite{MS}, 
because they only mention in \cite{bagno-chern} that the representation space is not the same.
\end{rmk}
It is also easy to describe directly the closure relation for the $B$-orbits in the \textsf{ANR} 
associated with $\ap_1$ for $\sono$ or $\sone$, i.e., in the setting of~\ref{subs:B} and~\ref{subs:D}(a).
We leave it as an exercise for the interested reader.

\subsection{Towards a general description of $B$-orbit closures}
Let $\sigma_\gamma$ be the reflection in $W$ corresponding to $\gamma\in \Delta^+$. If 
$\eus S\in\sah$, then all reflections $\sigma_\gamma$, $\gamma\in \eus S$, commute and
$\sigma_\eus S=\prod_{\gamma\in\eus S}\sigma_\gamma\in W$ is a well-defined involution. Let
$\textsf{Inv}(W)$ be the set of all involutions in $W$.  Associated with $\ah$, one obtains a subset 
$\textsf{Inv}(\ah):=\{\sigma_\eus S \mid \eus S \in\sah\} \subset \textsf{Inv}(W)$. 
Then one can suggest that numerical data of $\co_\eus S$ and $\co^*_\eus S$ are encoded in 
properties of $\sigma_\eus S$, or speculate that some properties of $\textsf{Inv}(\ah)$ 
are related to the closure of $B$-orbits in $\ah$  or $\ah^*$.
However, the involution $\sigma_\eus S\in W$ is one and the same for all abelian ideals $\ah$ such 
that $\eus S\in\sah$. But (the dimension of) the orbit $\co^*_\eus S$ depends on the choice of $\ah$, 
see Remark~\ref{rem:zavisit-ahs}. Therefore one cannot expect a general formula for 
$\dim\co^*_\eus S$ in terms of $\sigma_\eus S$. And we did not find any general relation between 
$\dim\co_\eus S$ and $\sigma_\eus S$, either.
However, our computations for small rank cases support the following special result:

 
\begin{conj}   \label{conj:bruhat-closure-ahs}
Let $\ah=\p^u$ be an \textsf{ANR}. For  the $B$-orbits in $\ah^*$, we have 
\begin{itemize}
\item[\sf (i)] \  $\co^*_{\eus S'}\curle  \co^*_\eus S$ \ if and only if \ $\sigma_{\eus S'}\le \sigma_{\eus S}$;
\item[\sf (ii)] \ $\displaystyle \dim \co^*_\eus S=\frac{\ell(\sigma_\eus S)+\rk(1-\sigma_\eus S)}{2}=
\frac{\ell(\sigma_\eus S)+\#\eus S}{2}$.
\end{itemize}
\end{conj}

\noindent
Recall from the Introduction that $\ell$ is the usual length function on $W$, and $\rk(1-\sigma_\eus S)$ is 
the rank of $1-\sigma_\eus S$ as endomorphism  of $\te$ (also known as the {\it absolute length} of 
$\sigma_\eus S\in W$). Therefore $\rk(1-\sigma_\eus S)=\#\eus S$. Since the posets 
$(\p^u/B,{\curle)}$ and $((\p^u)^*/B, \curle)$ are isomorphic,  isomorphism being given by the action of 
$w_{0,L}$ on $\mathfrak S_{\p^u}$ (Proposition~\ref{prop:isom-p-u}), Conjecture~\ref{conj:bruhat-closure-ahs} 
can be restated in terms of the $B$-orbits in  $\p^u$ as follows:

{\bf Conjecture'}.  \ {\it For the $B$-orbits in the \textsf{ANR} \ $\p^u$, we have 
\begin{itemize}
\item[\sf (i)'] \ $\co_{\eus S'}\curle  \co_\eus S$ \ if and only if \ 
$\sigma_{w_{0,L}(\eus S')}\le \sigma_{w_{0,L}(\eus S)}$\,;
\item[\sf (ii)'] \  $\displaystyle \dim \co_\eus S=
\frac{\ell(\sigma_{w_{0,L}(\eus S)})+\#\eus S}{2}.$
\end{itemize} }

\noindent
Recall that all abelian nilradicals are maximal abelian ideals (but not vice versa!). But 
Conjecture~\ref{conj:bruhat-closure-ahs} cannot be true for {\bf all} maximal abelian ideals.

\begin{ex}   \label{ex:D4}
For $\g=\mathfrak{so}_8$, there are four maximal abelian $\be$-ideals. Three of them are \textsf{ANR} 
of dimension $6$, and the fourth maximal ideal $\ah$ is 5-dimensional. For the standard choice of simple
roots $\Pi=\{\esi_1-\esi_2,\esi_2-\esi_3, \esi_3-\esi_4, \esi_3+\esi_4\}$, 
the corresponding set of roots is:
$\dah=\{\esi_1-\esi_4,\,\esi_1+\esi_4,\,\esi_1+\esi_3,\,\esi_2+\esi_3,\,\esi_1+\esi_2\}$.
Consider  
\[ 
\eus S=\min(\dah)=\{\esi_1-\esi_4,\,\esi_1+\esi_4,\,\esi_2+\esi_3\}. 
\]
Then $\eus S=\eus C^{l}$, i.e., $\dim\co_{\eus S}=5$, but we have 
$\dim\co^*_{\eus S}=3$ in the dual space. Here $\ell(\sigma_\eus S)=11$ and $\rk(1-\sigma_\eus S)=3$, but $3\ne (11+3)/2$.
The open $B$-orbit in $\ah^*$ correspond to $\eus C^u=\{\esi_1+\esi_2\}=\{\theta\}$ with
$\ell(\sigma_{\eus C^u})=\ell(\sigma_\theta)=9$. 
Therefore $\sigma_{\eus S}\not\le\sigma_{\eus C^u}$.
\end{ex}

\begin{rmk}
It is interesting that  $\textsf{Inv}(W)$ equipped with (the restriction of) the Bruhat 
order is a graded poset and the rank function is exactly 
$\sigma \mapsto \frac{\ell(\sigma)+\rk(1-\sigma)}{2}$, see \cite[Theorem\,4.8]{hult}. 
(For the classical cases, this was earlier proved by F.\,Incitti.) Furthermore, the integer 
$\ell(\sigma)+\rk(1-\sigma)$, $\sigma\in \textsf{Inv}(W)$, occurs in the study of spherical conjugacy 
classes in $G$, see \cite{3-it}, \cite{1-kit}.
This suggests that there might be more 
interesting relations between involutions of $W$ and $B$-orbits in abelian ideals.
\end{rmk}

\end{document}